\documentclass[a4paper,12pt,oneside,reqno]{amsart}

\usepackage[T1]{fontenc}
\usepackage[utf8]{inputenc}
\usepackage{lmodern}
\usepackage[english]{babel}

\usepackage[%
	left=2.5cm,       
	right=2.5cm,      
	top=3.5cm,        
	bottom=3.5cm,     
	heightrounded,    
	bindingoffset=0mm 
]{geometry}

\usepackage[pagewise]{lineno}
\usepackage{amsmath}
\usepackage{amssymb}
\usepackage{mathtools}
\usepackage{mathrsfs}
\usepackage[colorlinks=true,linkcolor=blue,urlcolor=blue,citecolor=blue]{hyperref}

\numberwithin{equation}{section}

\usepackage{hyperref} 
\usepackage[noabbrev,capitalize]{cleveref}

\usepackage[initials,nobysame]{amsrefs}

\usepackage{tikz}
\usepackage{graphicx}
\usepackage{xcolor}
\usepackage[font=small]{caption}

\usepackage{bookmark}

\theoremstyle{plain}
	\newtheorem{theorem}{Theorem}[section]
	\newtheorem{lemma}[theorem]{Lemma}
	\newtheorem{proposition}[theorem]{Proposition}
	\newtheorem{corollary}[theorem]{Corollary}

\theoremstyle{definition}
	\newtheorem{definition}[theorem]{Definition}
	\newtheorem{example}[theorem]{Example}
	\newtheorem{remark}[theorem]{Remark}
	\newtheorem{open.problem}[theorem]{Open Problem}

\newcommand{\Leb}[1]{\mathscr{L}^{#1}} 
\newcommand{\emb}{\hookrightarrow}

\newcommand{\N}{\mathbb{N}}
\newcommand{\R}{\mathbb{R}}
\newcommand{\M}{\mathcal{M}}
\newcommand{\eps}{\varepsilon}
\newcommand{\weakto}{\rightharpoonup}
\newcommand{\res}{\mathbin{\vrule height 1.6ex depth 0pt width
0.13ex\vrule height 0.13ex depth 0pt width 1.3ex}}

\title[Riemann-Liouville fractional Sobolev spaces]{A note on Riemann-Liouville fractional Sobolev spaces}
\author{Alessandro Carbotti$^1$ and Giovanni E. Comi$^2$}
\email{$^1$alessandro.carbotti@unisalento.it, $^2$giovanni.comi@uni-hamburg.de}
\thanks{$^1$Dipartimento di Matematica e Fisica, Universit\`a del Salento, Via Per Arnesano, 73100 Lecce, Italy.
$^2$Fachbereich Mathematik, Universit\"at Hamburg, Bundesstra{\normalfont\ss}e 55, 20146 Hamburg, Germany.}

\date{\today}  



\begin{document}

\begin{abstract}
Taking inspiration from a recent paper by Bergounioux et al., we study the Riemann-Liouville fractional Sobolev space $W^{s, p}_{RL, a+}(I)$, for $I = (a, b)$ for some $a, b \in \R, a < b$, $s \in (0, 1)$ and $p \in [1, \infty]$; that is, the space of functions $u \in L^{p}(I)$ such that the left Riemann-Liouville $(1 - s)$-fractional integral $I_{a+}^{1 - s}[u]$ belongs to $W^{1, p}(I)$. We prove that the space of functions of bounded variation $BV(I)$ and the fractional Sobolev space $W^{s, 1}(I)$ continuously embed into $W^{s, 1}_{RL, a+}(I)$. In addition, we define the space of functions with left Riemann-Liouville $s$-fractional bounded variation, $BV^{s}_{RL,a+}(I)$, as the set of functions $u \in L^{1}(I)$ such that $I^{1 - s}_{a+}[u] \in BV(I)$, and we analyze some fine properties of these functions. Finally, we prove some fractional Sobolev-type embedding results and we analyze the case of higher order Riemann-Liouville fractional derivatives.
\end{abstract}

\maketitle

\tableofcontents

\section{Introduction}

The goal of this paper is to analyze in detail the connection between some functional spaces defined through the Riemann-Liouville fractional operator and the classical Sobolev and $BV$ spaces on an interval $I = (a, b)$ of the real line.

The intuitive idea of defining a fractional version of the derivative and integral operators is as old as calculus itself, having being mentioned for the first time in an epistular exchange between Leibniz and de l'H${\rm \hat{o}}$pital which dates back to 1695 \cite{leibniz1849letter}. Fractional integrals and derivatives have proved to be useful in applications, since they arise naturally in many contexts such as viscoelasticity, neurobiology and finance, see for instance \cite{MR3488533, MR3557159, MR2379269, MR3089369, DV1}. Therefore, different examples of such operators are present in literature. Among these ones, Riemann-Liouville and Caputo fractional derivatives are the most exploited in the one-dimensional applications. 
Given a sufficiently smooth function $u$ on an interval $(a, b)$ and $s \in (0, 1)$, the left and right Riemann-Liouville $s$-fractional derivatives of $u$ are defined as
\begin{align*}
D^s_{a+}\left[u\right]\left(x\right) & :=\frac{d}{dx}\frac{1}{\Gamma(1-s)}\int_a^x \frac{u(t)}{(x-t)^s}dt, \\
D^s_{b-}\left[u\right]\left(x\right) & :=-\frac{d}{dx}\frac{1}{\Gamma(1-s)}\int_x^b \frac{u(t)}{(t-x)^s}dt,
\end{align*}
respectively, where $\Gamma$ is Euler's Gamma function. On the other hand, the left and right Caputo $s$-fractional derivatives of $u$ are set to be
\begin{align*}
	{}^C D^s_{a+}[u](x) &:= \frac{1}{\Gamma(1-s)}\int_a^x \frac{u'(t)}{(x-t)^s} \, dt, \\
	{}^C D^s_{b-}[u](x) & :=-\frac{1}{\Gamma(1-s)}\int_x^b \frac{u'(t)}{(t-x)^s} \, dt.
\end{align*}
It is easy to notice that Caputo $s$-fractional derivatives are just given by a commutation in the order of the operations of the left and right $(1 - s)$-fractional integrals,
\begin{equation*}
I^{1-s}_{a+}[u](x):=\frac{1}{\Gamma(1-s)}\int_a^x \frac{u(t)}{(x-t)^s} \, dt \ \text{ and } \ I^{1-s}_{b-}[u](x):= \frac{1}{\Gamma(1-s)}\int_x^b \frac{u(t)}{(t-x)^s} \,dt,
\end{equation*}
and the classical differentiation which define the Riemann-Liouville $s$-fractional derivatives. Indeed, it is possible to show that the difference between these two notion of fractional derivatives depends only on the values of $u$ on the endpoints $a, b$: more precisely, for any $u \in C^{1}([a, b])$ we have
\begin{align*}
D^s_{a+}[u](x) & ={}^C D^s_{a+}u(x)+\frac{u(a)}{\Gamma(1-s)}(x-a)^{-s}, \\
D^s_{b-}[u](x)& ={}^C D^s_{b-}u(x)+\frac{u(b)}{\Gamma(1-s)}(b-x)^{-s}.
\end{align*}
These relations can be used to derive, at least formally, an interesting relation between these notions of derivatives and the fractional Laplacian on the whole $\R$ (as it was done in \cite{CDV19}). Indeed, by sending $a \to - \infty$ and $b \to + \infty$, we see that the Riemann-Liouville and Caputo fractional derivatives coincide for functions in $C^{\infty}_{c}(\R)$ (or even $C^{1}_{c}(\R)$). Therefore, we can define the ``improper'' left and right Riemann-Liouville fractional derivatives of $u$ with fixed base points $\pm \infty$ as
	\begin{align*}
	D^{s}_{-\infty}[u](x) & := \frac{d}{dx} I^{1-s}_{-\infty}[u](x) = \frac{1}{\Gamma(1-s)} \int_{-\infty}^{x} \frac{u'(t)}{(x-t)^s} \, dt, \\
	D^{s}_{+\infty}[u](x) & := -\frac{d}{dx} I^{1 - s}_{+\infty}[u](x) = -\frac{1}{\Gamma(1-s)}\int_{x}^{+ \infty} \frac{u'(t)}{(t-x)^s} \, dt.
	\end{align*}
Then, using the equivalent Marchaud formulation (see for instance \cite[Section 13]{MR1347689}), we can prove that
	\begin{eqnarray*}
		D^s_{-\infty}[u](x)+D^s_{+\infty}[u](x)&=&
		\frac{s}{\Gamma(1-s)}\int_0^{+\infty}\frac{2u(x)-u(x+y)-u(x-y)}{y^{s+1}}dy \\
		&=&\frac{s}{2\Gamma(1-s)}\,\int_{-\infty}^{+\infty}\frac{2u(x)-u(x+y)-u(x-y)}{|y|^{s+1}}dy \\
		&=&\frac{s}{2 c_s \Gamma(1-s)}\left(-\Delta\right)^{\frac{s}{2}}u(x),
	\end{eqnarray*}
where $\left(-\Delta\right)^{\frac{s}{2}}$ denotes the fractional Laplacian of order $\frac{s}{2}$ and $$c_s:=\left(\int_{-\infty}^{+\infty}\frac{1-\cos(\omega)}{|\omega|^{s+1}}d\omega\right)^{-1}.$$
In addition, as it has been pointed out also in \cite[Section 1]{SS18}, it is easy to see that
\begin{equation*}
D^{s}_{-\infty}[u](x) - D^{s}_{+\infty}[u](x) = \frac{1}{\Gamma(1-s)} \int_{-\infty}^{+ \infty} \frac{u'(t)}{|x-t|^s} \, dt = \frac{1}{\Gamma(1-s) \nu_{1 - s}} I^{1 - s}[u'](x),
\end{equation*}
where 
\begin{equation*}
I^{\sigma}[v](x) := \nu_{\sigma} \int_{- \infty}^{+ \infty} \frac{v(t)}{|x - t|^{1 - \sigma}} \, dt, 
\end{equation*}
is the Riesz potential of order $\sigma \in (0, 1)$ of a function $v \in C^{\infty}_{c}(\R)$, and 
$$\nu_{\sigma} = \frac{\Gamma \left ( \frac{1 - \sigma}{2} \right )}{2^{\sigma} \sqrt{\pi} \Gamma \left ( \frac{\sigma}{2} \right )}.$$ 
This suggests that we may define a fractional derivative operator on the whole $\R$ for $u \in C^{\infty}_{c}(\R)$ by setting
\begin{equation*}
	\nabla^s u(x) := \mu_{s} \int_{- \infty}^{+\infty} \frac{u'(t)}{|x - t|^{s}} \, dt = \frac{\mu_{s}}{\nu_{1- s}} I^{1 - s}[u'](x), 
\end{equation*}
for some multiplicative constant $\mu_{s} > 0$, so that
	\begin{equation*}
	D^{s}_{-\infty}[u](x) - D^{s}_{+\infty}[u](x) = \frac{1}{\Gamma(1-s) \mu_{s}} \nabla^s u(x).
	\end{equation*}
This is indeed the one dimensional case of the new notion of fractional gradient on the whole $\R$, suitable also for a meaningful extension to $\R^{n}$, for all $n \ge 1$, which has been investigated in some recent papers, as \cites{CS18, CS19, SSS15, SS15, SS18, S19}, with the aim of generalizing classical vector calculus rules to the fractional setting, and thus providing a way to define weakly fractionally differentiable functions.

In the case of a bounded open interval $I = (a, b)$, the left (and right) Riemann-Liouville fractional derivatives for regular functions have been widely studied in the literature (we refer the interested reader to the monograph \cite{MR1347689} and the bibliography therein), while in recent years it has been considered the case of $L^{p}$-functions which admit left (and right) Riemann-Liouville fractional derivatives in a weak sense, thus defining the left (and right) Riemann-Liouville fractional Sobolev spaces $W^{s, p}_{RL, a+}(I)$ (and $W^{s, p}_{RL,b-}(I)$) \cites{BLNT, MR3311433, MR3144452}.
In this paper, we answer some questions posed in \cite{BLNT}; namely, we extend \cite[Theorem 4.1]{BLNT} from $SBV$ to $BV$ by proving that $BV(I)$, continuously embed into $W^{s,1}_{RL,a+}(I)$ (Theorem \ref{result:BV_W_s1_RL}). We actually show also that the embedding $W^{s, 1}(I) \hookrightarrow W^{s,1}_{RL,a+}(I)$ is continuous (Proposition \ref{prop:propopop}). The continuity of these embeddings can be useful in many variational models involving this kind of fractional operators. In addition, we advance the study of the spaces $W^{s, p}_{RL, a+}(I)$ by proving Sobolev-type embedding theorems (Theorem \ref{result:RLfracsobemb}) and by considering also the case $s > 1$. Furthermore, we introduce the space $BV^s_{RL,a+}(I)$ of the functions with left Riemann-Liouville $s$-fractional bounded variation; that is, functions in $L^1(I)$ with $(1-s)$-fractional integral in $BV(I)$. Then, we study some of its properties. For instance, we show that a function $u \in BV^{s}_{RL, a+}(I)$ belongs to $W^{s, 1}_{RL, a+}(I)$ if and only if its distributional left Riemann-Liouville fractional derivative $\mathcal{D}^{s}_{a+}[u]$ is absolutely continuous with respect to the one-dimensional Lebesgue measure $\Leb{1}$ (Proposition \ref{prop:BV_s_W_s_a}). In addition, we extend the case $p = 1$ of the fractional Sobolev-type embedding by proving that $BV^{s}_{RL, a+}(I) \emb L^{\frac{1}{1 - s}, \infty}(I)$ (Theorem \ref{thm:Sob_emb_BV}).

Since both left and right Riemann-Liouville fractional Sobolev spaces behave exactly in the same way for the results that we are interested into, we shall consider only the left ones. We specify further in the Remarks \ref{rem:reflection} and \ref{rem:leftneqright} how the two spaces are related and, with a simple counterexample, we show that they do not coincide.

The paper is structured in the following way. In Section \ref{sec:prel}, after having set some notation and recalled definitions and preliminary notions, we prove some representation formulas for Riemann-Liouville fractional Sobolev functions and duality relations involving the Caputo fractional derivative. Section \ref{sec:main} is devoted to the aforementioned embedding results, and, in addition, to the analysis of the asymptotics as $s\to 1^-$ of the Riemann-Liouville fractional derivative $D^s_{a+}[u]$ for a function $u \in BV(I)$. We also provide an extension result for the Riemann-Liouville fractional integral to the space of finite Radon measures on an open bounded interval $I$, and a counterexample to the embedding results in the case in which the interval is instead unbounded. 
In Section \ref{sec:BV_s} we define the space $BV^s_{RL,a+}(I)$ and we prove that it strictly contains $W^{s,1}_{RL,a+}(I)$ and hence, thanks to Theorem \ref{result:BV_W_s1_RL}, also $BV(I)$. Moreover, we show through an example that, despite the regularization properties of the fractional integral, the fractional derivative measure $\mathcal{D}^s_{a+}[u]$ of a function $u \in BV^s_{RL,a+}(I)$ does not enjoy any particular absolute continuity property in general, since it may involve Dirac delta measures.
In Section \ref{sec:Sob_act}, we study the continuity of the $(1-s)$-fractional integral in the Sobolev space $W^{1,p}(I)$ for $1\le p \le\infty$. As a corollary, we obtain the well known result on the inclusion relations between Riemann-Liouville fractional Sobolev spaces. In the case $p=\infty$, we show through a simple example that, if the function does not vanish in the initial point, its Riemann-Liouville fractional derivative cannot be essentially bounded, even if the function is locally analytic. We conclude the section with some results on the improved differentiability properties of $I^{s}_{a+}[u]$ for Sobolev functions $u \in W^{1, p}(I)$. Then, in Section \ref{sec:Sob_type_emb} we prove some fractional Sobolev-type embedding theorems for $W^{s, p}_{RL, a+}(I)$ and $BV^{s}_{RL, a+}(I)$.
In Section \ref{sec:higher} we extend some results obtained in the rest of the paper by taking into account higher order fractional derivatives; namely, we prove the continuity of the fractional integral between Sobolev spaces of greater integer order and the inclusion of the space of functions with bounded Hessian in a higher order Riemann-Liouville fractional Sobolev space. We conclude the work with some open questions in Section \ref{sec:open}.

\section{Notation and preliminaries} \label{sec:prel}

Through this paper we shall work on bounded open intervals $I = (a, b)$ in $\R$, for some $a, b \in \R$, $a < b$. Following the usual notation, the map $\Gamma : (0, \infty) \to (0, \infty)$ is {\em Euler's Gamma function}, see~\cite{A64}. As it is customary, we denote by $\M(V)$ the space of Radon measures on some Borel set $V \subset \R$, and we will consider mainly $\M(I)$ and $\M(\overline{I})$, where $\overline{I} = [a, b]$. We shall say that $\rho \in C^{\infty}_{c}((-1, 1))$ is a {\em standard mollifier} if $\rho \ge 0$, $\rho(x) = \rho(-x)$ and $\displaystyle \int_{-1}^{1} \rho \, dx = 1$. In addition, for all $\eps > 0$, we set $\displaystyle \rho_{\eps}(x) := \frac{1}{\eps} \rho \left ( \frac{x}{\eps} \right )$. For $k\in\N$, $h\in\N_0$, $p \in [1, \infty]$ and $\beta \in(0,1)$ we define the following spaces 
\begin{align*}
W^{k,p}_{a}(I) & := \left \{ u \in W^{k, p}(I) : u^{(j)}(a) = 0 \text{ for all } j \in \{0, \dots, k - 1\}\right \}, \\
C^{h,\beta}_a(\overline{I}) & := \left \{ u \in C^{h, \beta}(\overline{I}) : u^{(j)}(a) = 0 \text{ for all } j \in \{0, \dots, h \} \right \},
\end{align*} 
where 
\begin{equation*}
C^{h,\beta}(\overline{I}) := \{ u \in C^{h}(\overline{I}) : u^{(h)} \in C^{0, \beta}(\overline{I})\}.
\end{equation*}
We employ analogous definitions when the left endpoint $a$ is replaced with the right endpoint $b$.

For the convenience of the reader we recall in this section the definitions and some properties of a few well known functional spaces.

\begin{definition}
	Let $1\le p<\infty$. We say that a measurable function $u$ belongs to the {\em weak $L^{p}$-space} $L^{p,\infty}(I)$ if 
\begin{equation*}
\sup_{t>0} t^{p} \Leb{1}(\{x\in I\, :\ |u(x)|>t\})<\infty.
\end{equation*}
The function $$(0,+\infty)\ni t \to \Leb{1}(\{x\in I\, :\ |u(x)|>t\})$$ is called the {\em distribution function} of $u$. The space $L^{p, \infty}(I)$ is equipped with the quasi-norm
\begin{equation*}
\|u\|_{L^{p, \infty}(I)} := \sup_{t>0} t \, \Leb{1}(\{x\in I\, :\ |u(x)|>t\})^{\frac{1}{p}}.
\end{equation*}
\end{definition}

We recall a well known result on the embeddings of the weak $L^{p}$ spaces on sets with finite measure (see \cite[Exercise 1.1.11]{G14-C}).

\begin{lemma} \label{lem:embed_L_p_weak}
For all $1 \le r < p$ we have the continuous embeddings 
\begin{equation*}
L^{p}(I) \emb L^{p, \infty}(I) \emb L^{r}(I),
\end{equation*}
with the estimates
\begin{equation*}
\|u\|_{L^{p, \infty}(I)} \le \|u\|_{L^{p}(I)} \text{ for all } u \in L^{p}(I),
\end{equation*} 
and 
\begin{equation*} 
\|u\|_{L^{r}(I)} \le \left ( \frac{p}{p - r} \right )^{\frac{1}{r}} (b - a)^{\frac{1}{r} - \frac{1}{p}} \|u\|_{L^{p, \infty}(I)} \text{ for all } u \in L^{p, \infty}(I).
\end{equation*}
\end{lemma}

\subsection{$BV$ functions on the real line} 

\begin{definition}
Let $U$ be an open set in $\R$. We say that $u\in BV(U)$ if $u\in L^1(U)$ and its distributional derivative $Du$ is a finite Radon measure on $U$; that is, if there exists $\mu \in \M(U)$ such that 
\begin{equation*}
	\int_{U} u(x)\phi'(x)dx=-\int_{U} \phi(x)d\mu(x),
\end{equation*}
for all $\phi\in C_c^1(U)$, in which case we have $\mu=Du$.
\end{definition}

The space $BV(U)$ is a Banach space when equipped with the norm $$\left\|u\right\|_{BV(U)}:=\left\|u\right\|_{L^1(U)}+|Du|(U).$$ In addition, $BV$ functions on the real line are essentially bounded: we recall the statement in the case in which $U$ is a segment.

\begin{lemma} \label{result:BV_bounded}
We have $BV(I) \emb L^{\infty}(I)$ with a continuous embedding. In particular,
\begin{equation} \label{eq:BV_bounded}
\|u\|_{L^{\infty}(I)} \le \max\left\{1,\frac{1}{b-a}\right\}\|u\|_{BV(I)},
\end{equation}
for all $u \in BV(I)$.
\end{lemma}
\begin{proof}
Thanks to \cite[Proof of Lemma 5.21, Claim 3]{Evans_Gariepy}, we know that, for all $u \in BV(I)$ and $\Leb{1}$-a.e. $z \in I$,
\begin{equation*}
|u(z)| \le \frac{1}{b - a} \int_{a}^{b} |u(x)| \, dx + |Du|(I).
\end{equation*}
Hence, \eqref{eq:BV_bounded} follows immediately.
\end{proof}

As a consequence, it is not difficult to show that, if $u \in BV(I)$ and we set 
\begin{equation} \label{eq:extension_zero_BV}
\tilde{u}(x) = \begin{cases} u(x) & \text{if} \ x \in I, \\
0 & \text{if} \ x \in \R \setminus I,
\end{cases}
\end{equation}
then $\tilde{u} \in BV(\R)$. In addition, we may prove that, if $u\in BV(I)$, the approximate limits of $u$ in $a$ from the right, $u(a+)$, and in $b$ from the left, $u(b-)$, exist and they coincide with the precise representative of $2 \tilde{u}$ on those points. In other words, we have
\begin{equation*}
u(a+) := \lim_{r \to 0} \frac{1}{r} \int_{a}^{a + r} u(x) \, dx  \ \text{ and } \ u(b-) := \lim_{r \to 0} \frac{1}{r} \int_{b - r}^{b} u(x) \, dx,
\end{equation*}
so that, thanks to Lemma \ref{result:BV_bounded}, we obtain
\begin{equation} \label{eq:lim_ab_bound}
\max\{|u(a+)|, |u(b-)|\} \le \max\left\{1, \frac{1}{b - a}\right\} \|u\|_{BV(I)}.
\end{equation}
In addition, thanks to \cite[Corollary 3.80]{AFP} it is possible to see that, for any standard mollifier $\rho$, we have
\begin{equation} \label{eq:extreme_moll_conv}
(\rho_{\eps} \ast u)(a) \to \frac{u(a+)}{2} \text{ and } (\rho_{\eps} \ast u)(b) \to \frac{u(b-)}{2}.
\end{equation}
Finally, it is easy to notice that, consistently with \cite[Remark 4.1]{BLNT},
\begin{equation} \label{eq:grad_zero_extension}
D \tilde{u} = Du \res I + u(a+) \delta_{a} - u(b-)\delta_{b},
\end{equation}
where $\delta$ is the Dirac delta measure; while clearly $D \tilde{u} = 0$ in $\R \setminus \overline{I}$.

\begin{remark} \label{rem:Sobolev_AC_emb}
It is well known that $W^{1, p}(I) \emb BV(I)$ for all $p \in [1,\infty]$, so that, thanks to Lemma \ref{result:BV_bounded} and H\"older's inequality, for all Sobolev function $u \in W^{1, p}(I)$ we have
\begin{equation*}
\|u\|_{L^{\infty}(I)} \le \max\left\{1,\frac{1}{b-a}\right\} \|u\|_{BV(I)} \le \max\left\{(b -a),1\right\} (b-a)^{-\frac{1}{p}} \|u\|_{W^{1, p}(I)}.
\end{equation*}
As a consequence, the approximate limits of $u$ in $a$ from the right, $u(a+)$, and in $b$ from the left, $u(b-)$, exist and satisfy
\begin{equation*}
\max\{|u(a+)|, |u(b-)|\} \le \max\left\{(b -a),1\right\} (b-a)^{-\frac{1}{p}} \|u\|_{W^{1, p}(I)}.
\end{equation*}
In particular, Sobolev functions on an interval $I$ admit absolutely continuous representatives in $AC(\overline{I})$ (see \cite[Theorem 3.28 and Definition 3.31]{AFP} and the subsequent observations therein). Therefore, in the following we shall identify Sobolev functions with their absolutely continuous representatives, and write $u(a), u(b)$ instead of $u(a+), u(b-)$.
\end{remark}

Now, we recall some known facts from Measure Theory. If $\mu\in\mathcal{M}(I)$, then, by the Radon-Nikodym Theorem, we can split it into an absolutely continuous part (with respect to the Lebsegue measure) $\mu_{ac}$, and a singular part $\mu_s$, such that $\mu=\mu_{ac}+\mu_s$. Moreover, we can decompose the singular part $\mu_s$ into an atomic measure $\mu_j$ and a diffuse measure $\mu_c$; in this way, we have 
$$
\mu=\mu_{ac}+\mu_s=\mu_{ac}+\mu_j+\mu_c.
$$
In particular, this decomposition induces an analogous decomposition on $BV$ functions on the real line, which does not have a counterpart in higher dimensions. Namely, following \cite[Corollary 3.33]{AFP}, for any $u\in BV(I)$ we have
$$
u=u_{ac}+u_j+u_c,
$$
where $u_{ac} \in W^{1, 1}(I)$, $u_{j}$ is a jump function and $u_{c}$ is a Cantor function; that is, they satisfy
\begin{equation*}
(D u)_{ac} = u_{ac}' \Leb{1}, \ \ (Du)_{j} = Du_{j} \text{ and } (Du)_{c} = D u_{c}.
\end{equation*}

\subsection{Fractional Sobolev spaces on the real line}

We recall here the definition of Gagliardo-Slobodeckij fractional Sobolev space. For an exhaustive exposition of this theory, we refer the interested reader to \cite{MR2944369}.

\begin{definition}
	Let $s\in(0,1)$ and $p\in[1,\infty)$. We define the {\em fractional Sobolev space} $W^{s,p}(I)$ as
	$$
	W^{s,p}(I):=\left\{u\in L^p(I): (x, y) \to \frac{u(x)-u(y)}{|x-y|^{\frac{1}{p}+s}}\in L^p(I\times I)\right\}.
	$$
	We define the \emph{Gagliardo-Slobodeckij seminorm} of $u$ as
	$$
	[u]_{W^{s,p}(I)}:=\left(\int_a^b\int_a^b\frac{|u(x)-u(y)|^p}{|x-y|^{sp+1}}dxdy\right)^{1/p}.
	$$
\end{definition}

The space $W^{s,p}(I)$, endowed with the norm
$$
\left\|u\right\|_{W^{s,p}(I)}:= \|u\|_{L^p(I)}+[u]_{W^{s,p}(I)},
$$
is a Banach space, which is Hilbert when $p=2$ (see \cite{MR2944369}).

\begin{remark}
In the case $s>1$, $s=m+\sigma$ for some $m\in\N$ and $\sigma\in(0,1)$, we say that $u$ belongs to the fractional Sobolev space $W^{s,p}(I)$ if $u\in W^{m,p}(I)$ and $u^{(m)}\in W^{\sigma,p}(I)$.\end{remark}

We recall that the density of smooth compactly supported functions in $W^{s, p}(I)$ is ensured only in some cases.

\begin{theorem}[\cite{TesiLuca}, Theorem D.2.1.]
	\label{th:densityluca}
	Let $I$ a bounded open interval, $s\in(0,1)$ and $1\le p<\infty$ such that $sp<1$.
	
	Then, we have $\overline{C^\infty_c(I)}^{\left\|\cdot\right\|_{W^{s,p}(I)}}=W^{s,p}(I)$; that is, $C^\infty_c(I)$ is dense in $W^{s,p}(I)$.
\end{theorem}

\begin{remark}
	\label{rem:densitàdellecc1}
	As a byproduct of Theorem \ref{th:densityluca}, we have that also $C^1_c(I)$ is dense in $W^{s,1}(I)$.
\end{remark}

Now, we recall a fractional Hardy inequality introduced in \cite{MR2085428}. For the sake of simplicity, we state it only in the one dimensional case for open bounded intervals, though the result holds in any dimension and for any open bounded set with Lipschitz boundary; see also \cite[Theorem D.1.4]{TesiLuca} for a different proof.

\begin{lemma}{\cite[Theorem D.1.4]{TesiLuca}}
	\label{lem:hardy}
	Let $s\in(0,1)$, $p\in[1,\infty)$ such that $sp<1$ and $I=(a,b)$. Then, there exists $c=c(s,p,a,b)>0$ such that
	\begin{equation*}
	\int_a^b\frac{|u(x)|^p}{|\delta_I(x)|^{sp}}dx\leq c\left\|u\right\|^p_{W^{s,p}(I)} \text{ for all } u\in W^{s,p}(I),
	\end{equation*}
	where $|\delta_I(x)|:={\rm dist}(x,\partial I)=\min\{x-a,b-x\}$.
\end{lemma}

\subsection{Fractional integrals}

\begin{definition} \label{def:fract_int}
Let $u\in L^1\left(I\right)$ and $s\in\left(0,1\right)$. We define the {\em left and right Riemann-Liouville $s$-fractional integrals} as
\begin{equation} \label{eq:frac_int_1}
I^s_{a+}\left[u\right]\left(x\right):=\frac{1}{\Gamma\left(s\right)}\int_a^x \frac{u\left(t\right)}{\left(x-t\right)^{1-s}} \, dt,
\end{equation}
and
\begin{equation} \label{eq:frac_int_2}
I^s_{b-}\left[u\right]\left(x\right):=\frac{1}{\Gamma\left(s\right)}\int_x^b \frac{u\left(t\right)}{\left(t-x\right)^{1-s}}\, dt.
\end{equation}
\end{definition}

\begin{remark} \label{rem:frac_int_well_posed}
It is not difficult to check that definitions \eqref{eq:frac_int_1} and \eqref{eq:frac_int_2} are well posed for all $u \in L^{1}(I)$ and $s \in (0, 1)$. Indeed, we have
\begin{align*}
\|I^s_{a+}\left[|u|\right]\|_{L^{1}(I)} & = \frac{1}{\Gamma\left(s\right)} \int_{a}^{b} \int_{a}^{x} \frac{|u(t)|}{(x - t)^{1 - s}} \, d t \, dx = \frac{1}{\Gamma\left(s\right)} \int_{a}^{b} \int_{t}^{b} \frac{|u(t)|}{(x - t)^{1 - s}} \, d x \, dt \\
& = \frac{1}{s \Gamma\left(s\right)} \int_{a}^{b} |u(t)| (b - t)^{s} \, d t \le \frac{(b- a)^{s}}{\Gamma \left (s + 1 \right )} \|u\|_{L^{1}(I)},
\end{align*}
so that $I^s_{a+}\left[|u|\right] \in L^{1}(I)$, which implies $I^{s}_{a+}[u] \in L^1(I)$ with the same bound on the $L^1$-norm. In particular, $I^s_{a+}\left[u\right](x)$ is well defined for $\Leb{1}$-a.e. $x \in I$. A similar argument shows that also $I^s_{b-}\left[u\right] \in L^{1}(I)$, with
\begin{equation*}
\|I^s_{b-}\left[u\right]\|_{L^{1}(I)} \le \frac{(b- a)^{s}}{\Gamma \left (s + 1 \right )} \|u\|_{L^{1}(I)},
\end{equation*} 
so that $I^s_{b-}\left[u\right]$ is well defined almost everywhere in $I$.
\end{remark}

For the ease of the reader, we summarize in the following Propositions \ref{prop:cont_BLNT} and \ref{prop:holcontinuity} some results on the continuity properties of $I_{a+}^{s}$ presented in \cite[Section 3]{MR1347689} and \cite[Theorem 4]{MR1544927}. As a preliminary result, we recall here the known fact that $s$-Riemann-Liouville fractional integrals are continuous mappings from $L^{1}(I)$ into $L^{\frac{1}{1-s},\infty}(I)$, in analogy with the continuity properties of the Riesz potential of order $s$ on the whole $\R$, defined as
\begin{equation} \label{eq:Riesz_potential_R_def}
I^{s}[v](x) := \frac{\Gamma \left ( \frac{1 - s}{2} \right )}{2^{s} \sqrt{\pi} \Gamma \left ( \frac{s}{2} \right )} \int_{- \infty}^{+ \infty} \frac{v(t)}{|x - t|^{1 - s}} \, dt, 
\end{equation}
for $v \in L^{1}(\R)$, for which we refer to \cite[Chapter 5]{MR0290095} and \cite[Chapter 1]{G14}, for instance.
\begin{lemma}
	\label{lem:weakpq}
	Let $s\in(0,1)$. The fractional integral $I^s_{a+}$ is a weak type $\left (1,\frac{1}{1-s} \right )$ operator; namely, there exists $C_s > 0$ such that
\begin{equation} \label{eq:weakpq}
\Leb{1}\left(\left\{x\in I: |I^s_{a+}[u](x)|>t \right\}\right) \le C_s\left(\frac{\left\|u\right\|_{L^1(I)}}{t}\right)^{\frac{1}{1-s}} \text{ for all } u \in L^{1}(I).
\end{equation}
\end{lemma}
	\begin{proof}
		Given $u \in L^{1}(I)$, we set
		\begin{equation*}
		\tilde{u}(x):=\begin{cases}
		u(x) & \text{ if } \  x\in I \\
		0 & \text{ if } \ x\notin I,
		\end{cases}
		\end{equation*}
		and, for $c\in\R$, we set $H_c(x):=H(x-c)$, where $H$ denotes the Heaviside function $H(x):=\chi_{(0,+\infty)}(x)$.
Clearly, we have $$\left\|u\right\|_{L^1(I)}=\left\|\tilde{u}\right\|_{L^1(\R)}=\left\|\tilde{u}H_a\right\|_{L^1(\R)}.$$
		Now, we notice that, for $\Leb{1}$-a.e. $x\in I$,
		$$
		I^s_{a+}[u](x)=I^s_{a+}[\tilde{u}](x),
		$$
		and
\begin{equation} \label{eq:fract_int_conv_equiv}
		I^s_{a+}[\tilde{u}](x)=((\tilde{u}H_a)\ast(KH))(x),
\end{equation}
		where $\ast$ denotes the usual convolution operator on $\R$, and $\displaystyle K(y):=\frac{1}{\Gamma(s) |y|^{1-s}}$. In addition, for all $x\in \R$ we have
\begin{equation*}
|I^{s}_{a+}[\tilde{u}](x)| \le ((|\tilde{u}| H_{a}) \ast K)(x) = \frac{1}{\Gamma(s)} \int_{-\infty}^{+\infty} \frac{|\tilde{u}|(t) H_{a}(t)}{|x-t|^{1-s}} \, dt = c_{s} I^{s}(|\tilde{u}| H_{a})(x),
\end{equation*}
where $I^{s}$ is the Riesz potential of order $s$ on the whole $\R$, as defined in \eqref{eq:Riesz_potential_R_def}, and 
$$ c_{s} =  \frac{2^{s} \sqrt{\pi} \Gamma \left ( \frac{s}{2} \right )}{\Gamma(s) \Gamma \left ( \frac{1 - s}{2} \right )}.$$
		Thus, thanks to the weak type $\left (1,\frac{1}{1-s} \right)$ estimates for the Riesz potentials on the whole $\R$ (see \cite[Chapter 5, Theorem 1]{MR0290095} and \cite[Theorem 1.2.3]{G14}), there exists a constant $C_{s} > 0$ such that, for all $t>0$,
		\begin{align*}
		\Leb{1}\left(\left\{x\in I: |I^s_{a+}[u](x)|>t \right\}\right) & =\Leb{1}\left(\left\{x\in I: |I^s_{a+}[\tilde{u}](x)|>t\right\}\right) \\
		& \le\Leb{1}\left(\left\{x\in \R: c_{s}  I^{s}(|\tilde{u}| H_{a})(x) >t \right\}\right) \\
		& \le C_s\left(\frac{\left\|\tilde{u}H_a\right\|_{L^1(\R)}}{t}\right)^{\frac{1}{1-s}}=C_s\left(\frac{\left\|u\right\|_{L^1(I)}}{t}\right)^{\frac{1}{1-s}}.
		\end{align*}
This ends the proof.
		\end{proof}

\begin{remark}
	We notice that also $I^s_{b-}$ is a weak type $\left (1,\frac{1}{1-s}\right )$ operator. Indeed, using the same notation as in the proof of Lemma \ref{lem:weakpq}, it is enough to observe that
	$$
	I^s_{b-}[\tilde{u}]=(\tilde{u}H_b)\ast(KH_{-}),
	$$
	where $H_b(x):=H(b-x)$ and $H_{-}(x):=H(-x)$. Then the proof is analogous.
\end{remark}

\begin{proposition}[Continuity properties of the fractional integral in $L^p$ spaces] \label{prop:cont_BLNT}
Let $s\in(0,1)$. The fractional integral $I^s_{a+}$ is a continuous operator 
\begin{enumerate}
\item from $L^p(I)$ into $L^p(I)$ for all $p \in [1, \infty]$,
\item from $L^1(I)$ into $L^{\frac{1}{1-s},\infty}(I)$, and so into $L^r(I)$, for all $r\in \left [1,\frac{1}{1-s} \right )$,
\item from $L^p(I)$ into $L^r(I)$ for all $p\in \left (1,\frac{1}{s} \right )$ and $r\in \left [1,\frac{p}{1-sp} \right ]$,
\item from $L^p(I)$ into $C^{0,s-\frac{1}{p}}(\overline{I})$ for all $p \in \left (\frac{1}{s}, \infty \right )$,
\item from $L^{1/s}(I)$ into $L^r(I)$ for all $r\in[1,\infty)$,
\item from $L^{\infty}(I)$ into $C^{0,s}(\overline{I})$.
\end{enumerate}
\end{proposition}
\begin{proof}
Point (2) is a straightforward consequence of Lemmas \ref{lem:weakpq} and \ref{lem:embed_L_p_weak}, while the other points follow from \cite[Section 3]{MR1347689} and \cite[Theorem 4]{MR1544927}.
\end{proof}

\begin{remark} \label{rem:constant_Riesz_Thorin}
	We notice that point $(i)$ of Proposition \ref{prop:cont_BLNT} is a consequence of a generalized Minkowski inequality, as observed in the proof of \cite[Theorem 2.6.]{MR1347689}. Alternatively, we may use the fact that $I^s_{a+}$ is bounded from $L^1(I)$ into $L^1(I)$, thanks to \cref{rem:frac_int_well_posed}, and from $L^\infty(I)$ to $L^\infty(I)$, thanks to the trivial estimate 
	\begin{equation*}|I^s_{a+}[u](x)|\le\frac{(b-a)^s}{\Gamma(s+1)}\left\|u\right\|_{L^\infty(I)} \text{ for } \Leb{1}\text{-a.e. } x\in I,
	\end{equation*}
so that we can apply the Riesz-Thorin Theorem \cite[Theorem 1.3.4]{G14-C} to get the continuity from $L^p(I)$ into $L^p(I)$ for all $1<p<\infty$. In particular, since the constants of continuity from $L^1(I)$ into $L^{1}(I)$ and from $L^\infty(I)$ into $L^{\infty}(I)$ coincide, the Riesz-Thorin Theorem gives a bound for the constant of continuity $C_{s, p}$ from $L^p(I)$ into $L^{p}(I)$, namely,
	$$
	C_{s,p}\le \frac{(b-a)^s}{\Gamma(s+1)} \text{ for all } p \in [1, \infty].
	$$
This shows that $C_{s,p}$ is uniformly bounded in $p \in [1, \infty]$.
\end{remark}

\begin{corollary}\label{cor:bvintoholder}
	Let $s\in(0,1)$. The fractional integral $I^s_{a+}$ is a continuous operator from $BV(I)$ into $C^{0,s}(\overline{I})$.
	\begin{proof}
		The statement follows by combining Lemma \ref{eq:BV_bounded} and the last point of Proposition \ref{prop:cont_BLNT}.
		\end{proof}
\end{corollary}

 \begin{proposition}[Continuity properties of the fractional integral in H{\"o}lder spaces]\label{prop:holcontinuity}
	Let $s\in(0,1)$ and $\alpha\in\left(0,1\right]$. The fractional integral $I^{s}_{a+}$ is a continuous operator 
	\begin{enumerate}
		\item from $C^{0,\alpha}_{a}(\overline{I})$ onto $C^{0,\alpha+s}_{a}(\overline{I})$ if $\alpha+s<1$,
		\item from $C^{0,\alpha}_{a}(\overline{I})$ onto $H^{1,1}_{a}(\overline{I})$ if $\alpha+s=1$,
		\item from $C^{0,\alpha}_{a}(\overline{I})$ onto $C^{1,\alpha+s-1}_{a}(\overline{I})$ if $\alpha+s>1$,
	\end{enumerate}
where $H^{1,1}_{a}(\overline{I})$ is the space of functions $f \in C^{0}(\overline{I})$ that satisfy $f(a) = 0$ and admit $\omega(h)=|h||\log|h||$ as a local modulus of continuity; namely, for which there exists $C>0$ such that
$$
|f(x+h)-f(x)|\leq C|h||\log|h|| \text{ for all } h \in (a - x, b - x) \setminus \{0\} \text{ and } x \in \overline{I}.
$$
\end{proposition}
\begin{proof}
	We refer to \cite[Theorem 3.1]{MR1347689}, \cite[Lemma 13.1]{MR1347689} and the subsequent remark therein for the proof of the continuity of $I^s_{a+}$ in the cases $\alpha+s\le 1$ and $\alpha+s>1$ respectively. Then, we have that the operator $I_{a+}^{s}$ is also onto by \cite[Theorem 13.17]{MR1347689}, which actually holds for functions with a more general modulus of continuity and that vanish in the endpoint $a$.
	\end{proof}

We recall now the semigroup law, one of the most useful property of the fractional integrals, for which we refer to \cite[Section 2.3, formula (2.21)]{MR1347689} (see also \cite[Theorem 2.5]{MR1347689}).

\begin{lemma}[Semigroup law]
	\label{lem:semigrouplaw}
	Let $\alpha,\beta\in (0,1)$ such that $\alpha+\beta\le 1$ and $u\in L^1(I)$. Then, we have
	$$
	I^\alpha_{a+}[I^\beta_{a+}[u]]=I^{\alpha+\beta}_{a+}[u],
	$$
	where $\displaystyle I^1_{a+}[u](x):=\int_a^x u(t)dt$.
\end{lemma}

\begin{remark}
	\label{rem:reflection}
As stated in \cite[Section 2.3]{MR1347689}, the operators $I^s_{a+}$ and  $I^s_{b-}$ are related by a simple change of variable through the following formula
	$$
	I^s_{a+}[u](Q(x))=I^s_{b-}[u_Q(\cdot)](x),
	$$
	where $Q(x):=b+a-x$ and $u_Q(\cdot):=u(Q(\cdot))$.
\end{remark}

We recall now a simple duality relation between $I_{a+}^{s}$ and $I_{b-}^{s}$, which shall prove to be useful in the sequel.

\begin{lemma} \label{result:duality_fract_int}
Let $u, v \in L^{1}(I)$ and $s \in (0, 1)$. Then we have
\begin{equation} \label{eq:duality_fract_int}
\int_{a}^{b} I^{s}_{a+}[u](x) \, v(x) \, dx = \int_{a}^{b} u(x) \, I_{b-}^{s}[v](x) \, dx.
\end{equation}
\end{lemma}
\begin{proof}
By Fubini's Theorem, we have
\begin{align*}
\int_{a}^{b} I_{a+}^{s} [u](x) \, v(x) \, dx & = \frac{1}{\Gamma\left(s\right)} \int_{a}^{b} \int_{a}^{x} \frac{u(t)}{(x - t)^{1 - s}} \, v(x) \, dt \, dx  \\
& =  \frac{1}{\Gamma\left( s\right)} \int_{a}^{b} \int_{t}^{b} \frac{v(x)}{(x - t)^{1 - s}} \, u(t) \, dx \, dt \\
& = \int_{a}^{b} u(t) \, I_{b-}^{s}[v](t) \, dt.
\end{align*}
\end{proof}

We conclude this section by recalling a well known result on the convergence of $I_{a+}^{s}$ to the identity operator as $s \to 0^{+}$. 

\begin{lemma} \label{lem:fractional_integral_s_0}
For any $u \in L^{1}(I)$ we have $\|I_{a+}^{s}[u] - u\|_{L^{1}(I)} \to 0$ as $s \to 0^{+}$. In particular, if $u \in C^1(\overline{I})$, then $I_{a+}^{s}[u](x) \to u(x)$ for all $x \in I$ and it holds that
\begin{equation}
\label{eq:regularvolterra}
I^s_{a+}[u](x)=\frac{u(a)}{\Gamma(s+1)}(x-a)^s+\frac{1}{\Gamma(s+1)}\int_a^x u'(t)(x-t)^sdt.
\end{equation}
Analogous statements hold for $I_{b-}^{s}$.
\begin{proof}
	We start by assuming that $u\in C^1(\overline{I})$, then, thanks to a simple integration by parts, equality \eqref{eq:regularvolterra} immediately follows.
	Thus, letting $s\to 0^+$ we immediately obtain pointwise convergence, and by Lebesgue's dominated convergence theorem we have convergence in $L^1(I)$.
	Otherwise, if $u\in L^1(I)$, for any $\varepsilon>0$ there exists $v\in C^1(\overline{I})$ such that $$\|v-u\|_{L^1(I)}\le\varepsilon.$$ Then, by Remark \ref{rem:frac_int_well_posed}, we have
	\begin{equation*}
	\begin{split}
	\|I^s_{a+}[u]-u\|_{L^1(I)}&\leq \|I^s_{a+}[u-v]\|_{L^1(I)}+\|I^s_{a+}[v]-v\|_{L^1(I)}+\|v-u\|_{L^1(I)} \\
	&\leq\max\left\{1,\frac{(b-a)^s}{\Gamma(s+1)}\right\}\|v-u\|_{L^1(I)}+\|I^s_{a+}[v]-v\|_{L^1(I)} \\
	&\le\max\left\{1,\frac{(b-a)^s}{\Gamma(s+1)}\right\}\varepsilon+\|I^s_{a+}[v]-v\|_{L^1(I)}.
	\end{split}
	\end{equation*}
This implies that
\begin{equation*}
\limsup_{s \to 0^{+}} \|I^s_{a+}[u]-u\|_{L^1(I)} \le \eps,
\end{equation*}
so that the claim is proved, since $\eps$ is arbitrary.
	\end{proof}
\end{lemma}

\subsection{Fractional derivatives}

\begin{definition}[Riemann-Liouville fractional derivatives] \label{def:fract_der_RL}
Let $s \in (0, 1)$. For any $u: I \to \R$ sufficiently smooth, so that $I^{1 - s}_{a+}[u]$ and $I^{1 - s}_{b-}[u]$ are differentiable, we define the {\em left and right Riemann-Liouville $s$-fractional derivatives} of $u$ as
\begin{equation} \label{eq:frac_der_RL_1}
D^s_{a+}\left[u\right]\left(x\right):=\frac{d}{dx}I^{1-s}_{a+}\left[u\right]\left(x\right),
\end{equation}
and
\begin{equation} \label{eq:frac_der_RL_2}
D^s_{b-}\left[u\right]\left(x\right):=-\frac{d}{dx}I^{1-s}_{b-}\left[u\right]\left(x\right).
\end{equation}
\end{definition}

\begin{remark}
	As a consequence of the preceding Proposition \ref{prop:holcontinuity}, for $0<s<\alpha<1$ and $u\in C^{0,\alpha}_{a}(\overline{I})$, we have $I^{1-s}_{a+}[u]\in C^{1,\alpha-s}_{a}(\overline{I})$. Therefore, $\alpha$-H\"older continuity with $\alpha>s$ is a sufficient condition to ensure the existence of \eqref{eq:frac_der_RL_1} and \eqref{eq:frac_der_RL_2}.
\end{remark}

If one applies the Riemann-Liouville fractional integrals to the first derivative $u'$, whenever this operation makes sense, one has the following alternative definitions of left and right fractional derivatives.

\begin{definition}[Caputo fractional derivatives]
	Let $s\in(0,1)$. For any $u \in C^{1}(\overline{I})$ we define the {\em left and right Caputo $s$-fractional derivatives} of $u$ as
	\begin{equation}\label{eq:caputo}
	{}^C D^s_{a+}[u](x):=I^{1-s}_{a+}[u'](x)=\frac{1}{\Gamma(1-s)}\int_a^x \frac{u'(t)}{(x-t)^s}dt.
	\end{equation}
	\begin{equation}\label{eq:caputor}
	{}^C D^s_{b-}[u](x):=-I^{1-s}_{b-}[u'](x)=-\frac{1}{\Gamma(1-s)}\int_x^b \frac{u'(t)}{(t-x)^s}dt.
	\end{equation}
\end{definition}

We notice that the minimal functional spaces in which \eqref{eq:caputo} and \eqref{eq:caputor} are well defined are given by
\begin{equation}\begin{split}
	C^{1,s}_{a+}\,&:=\Big\{f:\overline{(a,+\infty)}\rightarrow\mathbb{R}
	: \; f\in AC\big(\overline{(a,t)}\big)\\&\qquad\quad \text{and}\quad
	\Theta_{s,f,t}\in L^1\big((a,t)\big),\; {\mbox{ for all }} t>a\Big\}.\end{split}
	\end{equation}
	and
\begin{equation}\begin{split}
C^{1,s}_{b-}\,&:=\Big\{f:\overline{(-\infty,b)}\rightarrow\mathbb{R} : \; f\in AC\big(\overline{(t,b)}\big)\\&\qquad\quad \text{and}\quad
\Psi_{s,f,t}\in L^1\big((t,b)\big),\; {\mbox{ for all }} t<b\Big\}.\end{split}
\end{equation}
where
\begin{equation*}(a,t)\ni\tau\,\mapsto\,\Theta_{s,f,t}(\tau):=
f^{'}(\tau)(t-\tau)^{-s}\end{equation*}
and
\begin{equation*}(t,b)\ni\tau\,\mapsto\,\Psi_{s,f,t}(\tau):=
f^{'}(\tau)(\tau-t)^{-s};\end{equation*}	
see {\it e.g.} \cite{CDV18} for more details about this fact. We remark that, in the notation of \cite{CDV18}, the space $C^{1,s}_{a+}$ is the space $C^{k,\beta}_a$ with $k=1$ and $\beta=s$, while the function $\Theta_{s,f,t}$ is $\Theta_{1,s,f,t}$.

For $u\in AC(\overline{I})$, a simple computation relates Riemann-Liouville and the Caputo fractional derivatives. Indeed, by formula \eqref{eq:regularvolterra} with $1-s$ in place of $s$ we have
\begin{align}
\label{eq:riluecap}
\frac{1}{\Gamma(1-s)}\int_{a}^{x} \frac{u(t)}{(x - t)^{s}} \, dt = \frac{u(a)}{\Gamma(2-s)}(x-a)^{1-s} + \frac{1}{\Gamma(2-s)} \int_{a}^{x} u'(t) (x - t)^{1 - s} \, dt.
\end{align}
Hence, differentiating in $x$ on both sides of \eqref{eq:riluecap} we obtain the following formula
\begin{equation}\label{eq:RLecap}
D^s_{a+}[u](x)={}^C D^s_{a+}u(x)+\frac{u(a)}{\Gamma(1-s)}(x-a)^{-s}.
\end{equation}
Analogously, for right derivatives we have
\begin{equation}\label{eq:RLecapr}
D^s_{b-}[u](x)={}^C D^s_{b-}u(x)+\frac{u(b)}{\Gamma(1-s)}(b-x)^{-s}.
\end{equation}
Therefore, Riemann-Liouville and Caputo fractional derivatives coincide for all $u\in AC(\overline{I})$ that vanish in the initial point $a$ for left derivatives, or in the final point $b$ for right derivatives.

We also notice that, if $u \in AC(\overline{I})$, we can exploit formula \eqref{eq:RLecap} to obtain another representation of the left Riemann-Liouville fractional derivative:
\begin{equation} \label{eq:Marchaud_equiv}
\begin{split}
D^s_{a+}[u](x)&=\frac{u(a)}{\Gamma(1-s)(x-a)^s}+\frac{1}{\Gamma(1-s)}\int_a^{x} \frac{u'(t)}{(x - t)^{s}}dt \\
&=\frac{u(a)}{\Gamma(1-s)(x-a)^s}+\frac{1}{\Gamma(1-s)}\int_a^{x}u'(t)\left(s\int_{x - t}^{x-a}\xi^{-s-1}d\xi+\frac{1}{(x-a)^s}\right)dt \\
& = \frac{u(x)}{\Gamma(1-s)(x-a)^s}+ \frac{s}{\Gamma(1-s)}\int_0^{x - a}\int_{x - \xi}^{x} u'(t) \xi^{-s-1} \, d t \, d \xi \\
& = \frac{u(x)}{\Gamma(1-s)(x-a)^s}+ \frac{s}{\Gamma(1-s)}\int_0^{x - a} \frac{u(x) - u(x - \xi)}{\xi^{s + 1}} \, d \xi \\
&=\frac{u(x)}{\Gamma(1-s)(x-a)^s}+\frac{s}{\Gamma(1-s)}\int_a^x\frac{u(x)-u(t)}{(x-t)^{s+1}}dt.
\end{split}
\end{equation}
This different representation formula of the left Riemann-Liouville fractional derivative is known as the \emph{left Marchaud fractional derivative}: 
$$
{}^M D^s_{a+}[u](x):=\frac{u(x)}{\Gamma(1-s)(x-a)^s}+\frac{s}{\Gamma(1-s)}\int_a^x\frac{u(x)-u(t)}{(x-t)^{s+1}}dt;
$$
For a precise treatment of this fractional differential operator, we refer to \cite{ferrari} and \cite{MR1347689}. \medskip

Now, we recall the notion of $L^p$-representability. 

\begin{definition}
	Let $1\le q\le \infty$ and $u\in L^q(I)$. We say that $u$ is $L^p$-representable if $u\in I_{a+}^s(L^p(I))$ or $u\in I^s_{b-}(L^p(I))$ for some $1\le p\le q$ and $s\in(0,1)$.
\end{definition}

From Proposition \ref{prop:cont_BLNT}, we see that, for all $1\le p\le \infty$ and $s\in(0,1)$, we have 
\begin{equation*}
I_{a+}^s(L^p(I)) \emb L^p(I).
\end{equation*} 
However, the above definition is nontrivial since the inclusion is strict, as it is shown by the following example.

\begin{example} \label{ex:critical_power}
Consider $$u(x):=\frac{(x-a)^{s-1}}{\Gamma(s)}$$ for some $s\in(0,1)$. Then we have $u\in L^p(I)$ for all $1\le p< \frac{1}{1-s}$, and, for all $x \in I$, we see that 
\begin{align}
I^{1-s}_{a+}[u](x) & =\frac{1}{\Gamma(1-s)\Gamma(s)}\int_a^x (t-a)^{s-1}(x-t)^{-s} \, dt = \frac{1}{\Gamma(1-s)\Gamma(s)} \int_0^1\sigma^{s-1}(1-\sigma)^{-s}d\sigma \nonumber \\
& = \frac{\beta(s, 1 - s)}{\Gamma(1-s)\Gamma(s)} = 1,  \label{eq:frac_int_critical_power}
\end{align}
by the properties of {\em Euler's beta function} $\beta$ (or {\em Euler's first integral}, see \cite{A64}). Therefore, we conclude that
\begin{equation} \label{eq:frac_der_critical_power}
D^{s}_{a+}[u](x) = 0 \ \text{ for all } \ x \in I,
\end{equation}
while the left Caputo $s$-fractional derivative is not well defined.
We prove now that the equation
\begin{equation}
\label{eq:norepresent}
I^s_{a+}[f]=u
\end{equation}
has no solution in $L^p(I)$. In fact, suppose by contradiction that there exists $f \in L^{p}(I)$ satisfying \eqref{eq:norepresent}. If we apply the $(1-s)$-fractional integral on both sides of \eqref{eq:norepresent}, thanks to Lemma \ref{lem:semigrouplaw} and \eqref{eq:frac_int_critical_power}, we get
\begin{equation*}
 \int_{a}^{x} f(t) \, dt = I^1_{a+}[f](x) = I^{1 - s}_{a+}[I^{s}_{a+}[f]](x) =I^{1-s}_{a+}[u](x)=1,
\end{equation*}
for all $x\in I$. Therefore, differentiating on both sides of the equation, we obtain $f=0$, which is clearly a contradiction. 
\end{example}

The next lemma gives a characterization of $L^p$-representability. We are going to state and prove it only in the case of left fractional integral, the other case being analogous.

\begin{lemma}[$L^p$-representability criterion]
	\label{lem:representcriteri}
	Let $q \in [1, \infty]$, $u\in L^q(I)$, $s\in(0,1)$ and $p \in [1, q]$. We have that $u\in I^s_{a+}(L^p(I))$ if and only if $I^{1-s}_{a+}[u]\in W^{1,p}(I)$ and $I^{1-s}_{a+}[u](a)=0$.
	\begin{proof}
		If $u\in I^s_{a+}(L^p(I))$, then $u=I^s_{a+}[f]$ for some $f\in L^p(I)$; therefore, using Lemma \ref{lem:semigrouplaw}, we get
		$$
		I^{1-s}_{a+}[u](x)=I^{1-s}_{a+}[I^s_{a+}[f]](x)=I^1_{a+}[f](x)=\int_a^x f(t)dt\in W^{1,p}(I),
		$$
		and $I^{1-s}_{a+}[u](a)=I^1_{a+}[f](a)=0$.
		On the other hand, if $I^{1-s}_{a+}[u]\in W^{1,p}(I)$ and $I^{1-s}_{a+}[u](a)=0$, then $I^{1-s}_{a+}[u]$ admits an absolutely continuous representative, its pointwise derivative $D^{s}_{a+}[u]$ is well defined $\Leb{1}$-a.e. and satisfies
		\begin{equation}
		\label{eq:rafeq}
		I^{1-s}_{a+}[u](x)=\int_a^x D^s_{a+}[u](t)dt = I^{1}_{a+}[D^{s}_{a+}[u]](x) =I^{1-s}_{a+}[I^s_{a+}[D^s_{a+}[u]]](x),
		\end{equation}
where we used Lemma \ref{lem:semigrouplaw} in the last equality.
Therefore, by applying $D^{1-s}_{a+}$ to both sides of the equation and exploiting \eqref{eq:rafeq}, we conclude that, for $\Leb{1}$-a.e. $x \in I$,
		\begin{align*}
		u(x) & =\frac{d}{dx} \int_{a}^{x} u(t) \, dt = \frac{d}{dx} I^{s}_{a+}[I^{1-s}_{a+}[u]](x) = D^{1 - s}_{a+}[I^{1-s}_{a+}[u]](x) \\
		& = D^{1 - s}_{a+}[I^{1-s}_{a+}[I^s_{a+}[D^s_{a+}[u]]]](x) = \frac{d}{dx} I^{s}_{a+}[I^{1-s}_{a+}[I^s_{a+}[D^s_{a+}[u]]]](x) \\
		& = \frac{d}{dx} \int_{a}^{x} I^s_{a+}[D^s_{a+}[u]](t) \, dt = I^s_{a+}[D^s_{a+}[u]](x),
\end{align*}
		with $D^s_{a+}[u]\in L^p(I)$, since $I^{1-s}_{a+}[u]\in W^{1,p}(I)$; so that $u\in I^s_{a+}(L^p(I))$, and this concludes the proof.
	\end{proof}
\end{lemma}

\subsection{Riemann-Liouville fractional Sobolev spaces}

We introduce now the left Riemann-Liouville fractional Sobolev spaces.

\begin{definition}[Riemann-Liouville fractional Sobolev spaces] \label{def:W_sp_RL}
Let $p \in [1, \infty]$ and $s\in\left(0,1\right)$. We define the {\em left Riemann-Liouville fractional Sobolev space} of order $s$ and summability $p$ as
\begin{equation} \label{eq:W_sp_RL}
W^{s,p}_{RL,a+}\left(I\right):=\left\{u\in L^p\left(I\right) : \ I^{1-s}_{a+}\left[u\right]\in W^{1,p}(I)\right\}.
\end{equation}
\end{definition}

\begin{remark}
Clearly, it is possible to define in an analogous way the {\em right Riemann-Liouville fractional Sobolev spaces} 
\begin{equation*}
W^{s, p}_{RL, b-}(I):=\left\{u\in L^p\left(I\right) : \ I^{1-s}_{b-}\left[u\right]\in W^{1,p}(I)\right\}.
\end{equation*} 
\end{remark}

\begin{remark}
	\label{rem:leftneqright}
	We notice that, if $u\in W^{s,1}_{RL,a+}(I)$, we have $u_Q\in W^{s,1}_{RL,b-}(I)$, thanks to Remark \ref{rem:reflection}. However, this does not necessarily imply that $u\in W^{s,1}_{RL,b-}(I)$. Indeed, let $I = (0, 1)$ and consider $$u(x):=\frac{x^{s-1}}{\Gamma(s)}.$$ 
	By Example \ref{ex:critical_power}, we have $I^{1-s}_{0+}[u](x)=1$ for any $x\in I$, and so $I^{1-s}_{0+}[u]\in W^{1,1}((0,1))$. On the other hand, we have 
	$$
	I^{1-s}_{1-}[u](x)=\frac{1}{\Gamma(1-s) \Gamma(s)}\int_x^1 t^{s-1}(t-x)^{-s}dt=\frac{1}{\Gamma(1-s)\Gamma(s)}\int_1^{\frac{1}{x}}\omega^{s-1}(\omega-1)^{-s}d\omega,
	$$
	and this function belongs to $L^1((0,1))\setminus W^{1,1}((0,1))$. It is easy to check the $L^1$-summability. On the other hand, for any $\varphi \in C^{1}_{c}(I)$, we have
\begin{align*}
\int_{0}^{1} I^{1-s}_{1-}[u](x) \varphi'(x) \, dx & = \frac{1}{\Gamma(1-s) \Gamma(s)} \int_{0}^{1} \int_{1}^{\frac{1}{x}}\omega^{s-1}(\omega-1)^{-s} \varphi'(x) \, d\omega \, dx \\
& = \frac{1}{\Gamma(1-s) \Gamma(s)} \int_{1}^{\infty} \int_{0}^{\frac{1}{\omega}} \omega^{s-1}(\omega-1)^{-s} \varphi'(x) \, dx \, d\omega \\
& = \frac{1}{\Gamma(1-s) \Gamma(s)} \int_{1}^{\infty} \omega^{s-1}(\omega-1)^{-s} \varphi \left (\frac{1}{\omega}\right) \, d\omega \\
& = \frac{1}{\Gamma(1-s) \Gamma(s)} \int_{0}^{1} \tau^{1 - s} (1 - \tau)^{-s} \tau^{s} \varphi(\tau) \, \frac{d \tau}{\tau^{2}} \\
& = \frac{1}{\Gamma(1-s) \Gamma(s)} \int_{0}^{1} \frac{1}{\tau (1 - \tau)^{s}} \varphi(\tau) \, d \tau,
\end{align*}
which means that the weak derivative of $I^{1-s}_{1-}[u](x)$ is 
$$ - \frac{1}{\Gamma(1-s) \Gamma(s)} \frac{1}{x (1 - x)^{s}},$$
and so we conclude that
	$$
	D^s_{1-}[u](x)=\frac{1}{\Gamma(1-s) \Gamma(s)}\frac{1}{x(1-x)^s}\notin L^1(I).
	$$
\end{remark}

Remark \ref{rem:reflection} actually shows that $u\in W^{s,p}_{RL,a+}(I)$ if and only if $u_{Q} \in W^{s,p}_{RL,b-}(I)$, where $Q(x)=a+b-x$, though Remark \ref{rem:leftneqright} clarifies that $W^{s,p}_{RL,a+}(I)\neq W^{s,p}_{RL,b-}(I)$, in general. Nevertheless, since our results are analogous both for left and for right fractional integrals and derivatives, from this point on we shall work with the left Riemann-Liouville fractional Sobolev spaces.

It is not difficult to see that the space $W^{s,p}_{RL,a+}\left(I\right)$, endowed with the norm
\begin{equation}
\label{eq:norminrimliu}
\left\|u\right\|_{W^{s,p}_{RL,a+}(I)}:=\left\|u\right\|_{L^p(I)}+\left\|I^{1-s}_{a+}[u]\right\|_{W^{1,p}(I)},
\end{equation}
is a Banach space.

In the light of Definition \ref{def:W_sp_RL}, we may rephrase Lemma \ref{lem:representcriteri} in the following way.

\begin{lemma} \label{lem:image_fract_int_Sob}
Let $s\in(0,1)$ and $p \in [1, \infty]$. Then, $u\in I^s_{a+}(L^p(I))$ if and only if $u \in W^{s,p}_{RL, a+}(I)$ and $I^{1-s}_{a+}[u](a)=0$. 
\end{lemma}

We consider now a version of the fundamental Theorem of Calculus for left Riemann-Liouville fractional derivatives. A similar result was stated in \cite[Proposition 5]{MR3311433}, however we provide here a short proof, for completeness.

\begin{lemma}\label{lem:invinteder}
Let $s\in(0,1)$ and $u \in L^{1}(I)$. Then, for $\Leb{1}$-a.e. $x \in I$, we have
\begin{equation}\label{eq:invinteder_1}
		u(x)=D^s_{a+}[I^s_{a+}[u]](x).
\end{equation}
If $u\in W^{s,1}_{RL,a+}(I)$, then, for $\Leb{1}$-a.e. $x\in I$, we also have 
\begin{equation}\label{eq:invinteder_2}
		u(x) = I^s_{a+}[D^s_{a+}[u]](x)+\frac{I^{1-s}_{a+}[u](a)}{\Gamma(s)}(x-a)^{s-1}.
\end{equation}
Finally, if $u\in W^{s,1}_{RL,a+}(I) \cap I^s_{a+}(L^1(I))$, then \begin{equation}\label{eq:invinteder_3}
		u(x) = D^s_{a+}[I^s_{a+}[u]](x) = I^s_{a+}[D^s_{a+}[u]](x) \ \text{ for } \Leb{1}\text{-a.e. } x \in I.
\end{equation}
\end{lemma}
	\begin{proof}
If $u\in L^1(I)$, we have $I^s_{a+}[u]\in L^1(I)$, by Remark \ref{rem:frac_int_well_posed}, and, by Lemma \ref{lem:semigrouplaw},
			$$
			I^{1-s}_{a+}[I^s_{a+}[u]](x)=I^1_{a+}[u](x)=\int_a^x u(t)dt\in W^{1,1}(I).
			$$
Therefore, for $\Leb{1}$-a.e. $x \in I$, we get
			$$
			D^s_{a+}[I^s_{a+}[u]](x) = \frac{d}{dx} I^{1-s}_{a+}[I^s_{a+}[u]](x) = \frac{d}{dx}(I^1_{a+}[u](x)) = u(x).
			$$
In order to prove \eqref{eq:invinteder_2}, we notice that $I^{1 - s}_{a+}[u] \in W^{1, 1}(I)$ with weak derivative $D^{s}_{a+}[u] \in L^{1}(I)$, so that, for $\Leb{1}$-a.e. $x \in I$, 
			\begin{equation*}
			\begin{split}
			I^{1-s}_{a+}[u](x)&=\int_a^x D^s_{a+}[u](t)dt+I^{1-s}_{a+}[u](a) \\
			&=I^{1-s}_{a+}[I^s_{a+}[D^s_{a+}[u]]](x)+I^{1-s}_{a+}\left[\frac{I^{1-s}_{a+}[u](a)}{\Gamma(s)}(\cdot-a)^{s-1}\right](x),
			\end{split}
			\end{equation*}
by Lemma \ref{lem:semigrouplaw} and \eqref{eq:frac_int_critical_power}. We notice that, by Remark \ref{rem:frac_int_well_posed}, $I^s_{a+}[D^s_{a+}[u]] \in L^1(I)$, since $D^{s}_{a+}[u] \in L^{1}(I)$ by assumption. Therefore, we apply $D^{1-s}_{a+}$ to both sides of the equation and use \eqref{eq:invinteder_1} to obtain \eqref{eq:invinteder_2}. 
Finally, if $u\in W^{s,1}_{RL,a+}(I) \cap I^s_{a+}(L^1(I))$, then, by Lemma \ref{lem:representcriteri} with $p=q=1$, we have $I^{1-s}_{a+}[u](a)=0$, and this ends the proof.
	\end{proof}

We notice that, in an analogous way, it is possible to find a version of the Fundamental Theorem of Calculus for right Riemann-Liouville derivatives (see \cite[Proposition 6]{MR3311433}).

\begin{remark}
	It is worth noticing that these equalities are stable when $s\to 1^-$ for $u\in C^1(\overline{I})$. Indeed, we have 
	\begin{equation*}
	\begin{split}
	u(x)&=\lim_{s\to 1^-}D^s_{a+}[I^s_{a+}[u]](x)=\frac{d}{dx}\left(\int_a^x u(t)dt\right)=\int_a^x u'(t)dt+u(a) \\
	&=\lim_{s\to 1^-}I^s_{a+}[D^s_{a+}[u]](x)+\frac{I^{1-s}_{a+}[u](a)}{\Gamma(s)}(x-a)^{s-1},
	\end{split}
	\end{equation*}
	where the second equality exploits Lemma \ref{lem:fractional_integral_s_0}.
\end{remark}

We point out that there is a duality relation between the left Riemann-Liouville fractional derivative and the Caputo right fractional derivative, as shown in the following lemma.

\begin{lemma} \label{result:intbp}
	Let $u\in W^{s,1}_{RL,a+}(I)$, $v\in C^1_c(I)$ and $s \in (0, 1)$. Then we have
	\begin{equation} \label{eq:intbp}
	\int_{a}^{b} D^s_{a+}[u](x) \, v(x) \, dx = \int_{a}^{b} u(x) \, {}^CD_{b-}^s[v](x) \, dx.
	\end{equation}
	\begin{proof}
		Integrating by parts, and using Fubini's theorem, we have
		\begin{align*}
		\int_a^b D^s_{a+}[u](x)v(x)dx &=-\int_{a}^{b} I_{a+}^{1 - s} [u](x) \, v'(x) \, dx  \\
		&= -\frac{1}{\Gamma\left(1-s\right)} \int_{a}^{b} \int_{a}^{x} \frac{u(t)}{(x - t)^{s}} \, v'(x) \, dt \, dx  \\
		& = -\frac{1}{\Gamma\left(1-s\right)} \int_{a}^{b} \int_{t}^{b} \frac{v'(x)}{(x - t)^{s}} \, u(t) \, dx \, dt \\
		& = \int_{a}^{b} u(t) \, {}^CD_{b-}^s[v](t) \, dt.
		\end{align*}
	\end{proof}
\end{lemma}

We notice that, in light of the continuity of the fractional integral in $L^p$ given by Proposition \ref{prop:cont_BLNT}, the norm in \eqref{eq:norminrimliu} is equivalent to the one given by
$$
\left\|u\right\|:=\left\|u\right\|_{L^p(I)}+\left\|D^s_{a+}[u]\right\|_{L^p(I)}.
$$
Therefore, one could define the space $W^{s,p}_{RL,a+}(I)$ simply requiring that $u\in L^p(I)$ has a weak fractional derivative in $L^p(I)$. This would mean that there exists a function $w \in L^{p}(I)$ such that
\begin{equation*}
\int_{a}^{b} u(x) \, {}^CD_{b-}^s[v](x) \, dx = \int_{a}^{b} w(x) \, v(x) \, dx, 
\end{equation*}
for any $v \in C^{1}_{c}(I)$, in analogy with the duality formula \eqref{eq:intbp}.

\section{Main embedding and asymptotic results} \label{sec:main}

We start with a technical result concerning the action of the fractional integral on $\M(I)$. In what follows, for any $\mu \in \M(I)$ and $x \in I$, we shall use the notation $\displaystyle \int_a^x  f \, d\mu$ to denote the integral on the open interval $(a,x)$ of some Borel measurable function $f$. This choice is justified by the fact that $|\mu|\left(\{x\}\right)=0$ for all but countably many $x\in(a,b)$, thanks to the nonconcentration properties of Radon measures. As a consequence, there is no ambiguity when integrating the function $\displaystyle x \to \int_a^x  f \, d\mu$ in $dx$ over $I$.

\begin{proposition} \label{result:measure_fract_int}
Let $s \in (0, 1)$. The map $I^{s}_{a+}$ can be continuously extended to a map from $\mathcal{M}(I)$ into $L^{1}(I)$, by setting 
\begin{equation*}
I^{s}_{a+}[\mu](x):= \frac{1}{\Gamma\left(s\right)} \int_a^x \frac{d\mu(t)}{(x-t)^{1 - s}} 
\end{equation*}
for $\mu \in \M(I)$.
Then, $I^{s}_{a+}$ satisfies the following bound:
\begin{equation} \label{eq:measure_frat_int}
\left\|I_{a+}^{s}[\mu]\right\|_{L^{1}(I)} \le \frac{(b - a)^{s}}{\Gamma\left(1 + s\right)}|\mu|(I),
\end{equation}
for all $\mu \in \M(I)$.
\begin{proof}
	Since the function $(x - t)^{s - 1}$ is continuous in $t \in (a, x)$, for any fixed $x \in (a, b)$, the integral of this function against any nonnegative measure $\mu \in \M(I)$ is well defined, and we set
	\begin{equation*}
	I^{s}_{a+}[\mu](x):= \frac{1}{\Gamma\left(s\right)} \int_a^x\frac{d\mu(t)}{(x-t)^{1 - s}}.
	\end{equation*}
	Then, a simple computation similar to the one in Remark \ref{rem:frac_int_well_posed} shows that
	\begin{align*}
	\left\|I_{a+}^{s}[\mu]\right\|_{L^{1}(I)}&=\int_a^b\left| I_{a+}^{s}[\mu](x)\right|dx = \frac{1}{\Gamma\left(s\right)} \int_a^b\left|\int_a^x \frac{d\mu(t)}{(x-t)^{1 - s}}\right|dx=\frac{1}{\Gamma(s)}\int_a^b \int_a^x\frac{d \mu(t)}{(x-t)^{1 - s}} \, dx \\
	&= \frac{1}{\Gamma\left(s\right)} \int_a^b d \mu (t)\int_t^b\frac{dx}{(x-t)^{1 - s}}= \frac{1}{s \Gamma\left(s\right)} \int_a^b (b-t)^{s}d \mu(t) \\
	&\leq \frac{(b - a)^{s}}{\Gamma\left(1 + s\right)}\int_a^b d \mu(t)=\frac{(b - a)^{s}}{\Gamma\left(1 + s\right)} \mu(I).
	\end{align*}
	In the general case of $\mu \in \M(I)$, we consider the Jordan decomposition $\mu = \mu^{+} - \mu^{-}$ and we set
	\begin{equation*}
	I^{s}_{a+}[\mu](x) := I^{s}_{a+}[\mu^{+}](x) - I^{s}_{a+}[\mu^{-}](x) = \frac{1}{\Gamma\left(s\right)} \int_a^x\frac{d\mu(t)}{(x-t)^{1 - s}},
	\end{equation*}
by the linearity of the integral.
	Therefore, arguing as above, for any $\mu\in\mathcal{M}(I)$ we get
	$$
	\left\|I_{a+}^{s}[\mu]\right\|_{L^{1}(I)}\leq\frac{1}{\Gamma(s)}\int_a^b \int_a^x\frac{d|\mu|(t)}{(x-t)^{1 - s}} \, dx \leq \frac{(b - a)^{s}}{\Gamma\left(1 + s\right)}|\mu|(I),
	$$
	which ends the proof.
\end{proof}
\end{proposition}
It is not difficult to see that \cref{result:duality_fract_int} can be extended to couples of measures and essentially bounded functions.

\begin{lemma} \label{result:duality_fract_int_meas}
Let $\mu \in \M(I)$, $\phi \in L^{\infty}(I)$ and $s \in (0, 1)$. Then we have
\begin{equation} \label{eq:duality_fract_int_meas}
\int_{a}^{b} I^{s}_{a+}[\mu](x) \, \phi(x) \, dx = \int_a^b \, I_{b-}^{s}[\phi](x) \, d\mu(x).
\end{equation}
\end{lemma}
\begin{proof}
Notice that, by Proposition \ref{prop:cont_BLNT}, $I_{b-}^{s}[\phi] \in C^{0, s}(\overline{I})$, so that it is continuous and bounded, in particular. This implies that the integral on the right hand side of \eqref{eq:duality_fract_int_meas} is well defined. In addition, notice that
\begin{align*}
\int_{a}^{b} \int_a^x \frac{|\phi(x)|}{(x - t)^{1 - s}} \, d |\mu|(t) \, dx & \le \|\phi\|_{L^{\infty}(I)} \int_a^b \int_{t}^{b} (x - t)^{s - 1} \, dx \, d |\mu|(t) \\
& \le \|\phi\|_{L^{\infty}(I)} \frac{(b - a)^{s}}{s} |\mu|(I) < \infty.
\end{align*} 
Therefore, we may apply Fubini's theorem, and we obtain
\begin{align*}
\int_{a}^{b} I_{a+}^{s} [\mu](x) \, \phi(x) \, dx & = \frac{1}{\Gamma\left(s\right)} \int_{a}^{b} \int_a^x \phi(x) \, \frac{d\mu(t)}{(x - t)^{1 - s}} \, dx  =  \frac{1}{\Gamma\left( s\right)} \int_a^b \int_{t}^{b} \frac{\phi(x)}{(x - t)^{1 - s}} \, dx \, d \mu(t) \\
& = \int_{a}^{b} I_{b-}^{s}[\phi](t) \, d \mu(t).
\end{align*}
\end{proof}

Another interesting consequence of Proposition \ref{result:measure_fract_int} is that we can generalize Lemma \ref{lem:weakpq}, by proving the continuity of $I^{s}_{a+}$ from $\mathcal{M}(I)$ to $L^{\frac{1}{1 - s}, \infty}(I)$.

\begin{lemma} \label{lem:weak_meas}
Let $s\in(0,1)$. Then $I^s_{a+}$ maps continuously $\mathcal{M}(I)$ in $L^{\frac{1}{1 - s}, \infty}(I)$; that is, there exists $C_s > 0$ such that
\begin{equation*} 
\Leb{1}\left(\left\{x\in I: |I^s_{a+}[\mu](x)|>t \right\}\right) \le C_s\left(\frac{|\mu|(I)}{t}\right)^{\frac{1}{1-s}} \text{ for all } \mu \in \mathcal{M}(I).
\end{equation*}
\end{lemma}
	\begin{proof}
Given $\mu \in \mathcal{M}(I)$, we denote by $\tilde{\mu}$ its zero extension to the whole $\R$; that is, the measure defined by $$\tilde{\mu}(B) = \mu(B \cap I) \ \text{ for all Borel sets } B \subset \R.$$
As in the proof of Lemma \ref{lem:weakpq}, we set $H(x):=\chi_{(0,+\infty)}(x)$ and $\displaystyle K(x):=\frac{1}{\Gamma(s) |x|^{1-s}}$.	
It is then obvious that	
\begin{equation*}
|\mu|(I) = |\tilde{\mu}|(\R).
\end{equation*}
It is also clear that, for $\Leb{1}$-a.e. $x\in I$,
\begin{equation} \label{eq:id_tilde_mu_fract_int}
		I^s_{a+}[\mu](x)=I^s_{a+}[\tilde{\mu}](x) = (\tilde{\mu}\ast(K H))(x).
\end{equation}
Let now $\rho \in C^{\infty}_{c}((-1, 1))$ be a standard mollifier. Thanks to \cite[Theorem 2.2]{AFP}, we have
\begin{equation} \label{eq:conv_tot_var_easy_est}
\|\rho_{\eps} \ast \tilde{\mu}\|_{L^{1}(\R)} \le |\tilde{\mu}|(\R) = |\mu|(I),
\end{equation}
while, by \eqref{eq:fract_int_conv_equiv} and the standard associativity properties of convolution, for all $x \in I$ we get
\begin{align} 
I^{s}_{a+}[\rho_{\eps} \ast \tilde{\mu}](x) & = ( (H_{a}(\rho_{\eps} \ast \tilde{\mu})) \ast (K H))(x) \nonumber \\
& = (\rho_{\eps} \ast \tilde{\mu} \ast (K H))(x) + ( ((H_{a} - 1)(\rho_{\eps} \ast \tilde{\mu}))\ast (K H))(x) \nonumber \\
& = (\rho_{\eps} \ast I_{a+}^{s}[\tilde{\mu}])(x) - ( ( \chi_{(a - \eps, a)}(\rho_{\eps} \ast \tilde{\mu}))\ast (K H))(x), \label{eq:ass_prop_conv}
\end{align}
since $(\rho_{\eps} \ast \tilde{\mu})(x) = 0$ for all $x < a - \eps$, and $(1 - \chi_{(a, +\infty)}) \chi_{(a - \eps, +\infty)} = \chi_{(a - \eps, a)}$.
By \eqref{eq:weakpq} and \eqref{eq:conv_tot_var_easy_est}, there exists $C_s > 0$ such that, for all $\eps, t>0$, we get
\begin{equation*}
\Leb{1}\left(\left\{x\in I: |I^s_{a+}[\rho_{\eps} \ast \tilde{\mu}](x)|>t \right\}\right) \le C_s\left(\frac{\|\rho_{\eps} \ast \tilde{\mu}\|_{L^{1}(I)}}{t}\right)^{\frac{1}{1-s}} \le C_{s} \left ( \frac{|\mu|(I)}{t} \right )^{\frac{1}{1-s}}
\end{equation*}
Finally, we recall that there exists a suitable subsequence $\eps_{k} \to 0$ such that 
\begin{equation*}
(\rho_{\eps_{k}} \ast I^{s}_{a+}[\tilde{\mu}])(x) \to I^{s}_{a+}[\tilde{\mu}](x) = I^{s}_{a+}[\mu](x) \ \text{ for }\Leb{1}\text{-a.e. } x \in I,
\end{equation*} 
since $I^{s}_{a+}[\mu] \in L^{1}(I)$, by Proposition \ref{result:measure_fract_int}, and by \eqref{eq:id_tilde_mu_fract_int}. As for the term $( \chi_{(a - \eps, a)}(\rho_{\eps} \ast \tilde{\mu}))\ast (K H)$, we notice that it converges to zero in $L^1(I)$, since 
\begin{align*}
\|( \chi_{(a - \eps, a)}(\rho_{\eps} \ast \tilde{\mu}))\ast (K H)\|_{L^{1}(I)} & \le \int_{a}^{b} \int_{\R} \chi_{(a - \eps, a)}(y) (\rho_{\eps} \ast |\tilde{\mu}|)(y) \frac{\chi_{(- \infty, x)}(y)}{\Gamma(s) |x - y|^{1 - s}} \, dy dx \\
& = \int_{a - \eps}^{a} \int_{a}^{b} (\rho_{\eps} \ast |\tilde{\mu}|)(y) \frac{1}{\Gamma(s) (x - y)^{1 - s}} \, dx dy \\
& = \frac{1}{\Gamma(s + 1)} \int_{a - \eps}^{a} (\rho_{\eps} \ast |\tilde{\mu}|)(y) \left ( (b - y)^{s} - (a - y)^{s} \right ) \, dy \\
& \le \frac{(b - a + \eps)^{s}}{\Gamma(s + 1)} \int_{a - \eps}^{a} \int_{y - \eps}^{y + \eps} \rho_{\eps}(z - y) \, d |\tilde{\mu}|(z) \, dy \\
& = \frac{(b - a + \eps)^{s}}{\Gamma(s + 1)} \int_{a - 2\eps}^{a + \eps} \int_{z - \eps}^{z + \eps} \rho_{\eps}(z - y) \, \, dy \, d |\tilde{\mu}|(z)  \\
& = \frac{(b - a + \eps)^{s}}{\Gamma(s + 1)} |\tilde{\mu}|((a - 2 \eps, a + \eps)) \\
& = \frac{(b - a + \eps)^{s}}{\Gamma(s + 1)} |\mu|((a, a + \eps)) \to 0
\end{align*}
as $\eps \to 0$. Hence, up to passing to another subsequence, we obtain that
\begin{align*}
I^{s}_{a+}[\rho_{\eps_{k}} \ast \tilde{\mu}](x) & = (\rho_{\eps_{k}} \ast I_{a+}^{s}[\tilde{\mu}])(x) - ( ( \chi_{(a - \eps_{k}, a)}(\rho_{\eps_{k}} \ast \tilde{\mu}))\ast (K H))(x) \\
& \to I^{s}_{a+}[\mu](x) \text{ for } \Leb{1}\text{-a.e. } x \in I.
\end{align*}
Hence, exploiting the lower semicontinuity of the distribution function (see \cite[Exercise 1.1.1]{G14-C}) and \eqref{eq:ass_prop_conv}, we get
\begin{align*}
\Leb{1}\left(\left\{x\in I: |I^s_{a+}[\mu](x)|>t \right\}\right) & \le \liminf_{k \to \infty} \Leb{1}\left(\left\{x\in I: |I^s_{a+}[\rho_{\eps_{k}} \ast \tilde{\mu}](x)|>t \right\}\right) \\
& \le C_{s} \left ( \frac{|\mu|(I)}{t} \right )^{\frac{1}{1-s}},
\end{align*}
and this ends the proof.
\end{proof}

We notice that, as a byproduct of the proof of \cite[Theorem 3.3]{BLNT}, formula \eqref{eq:RLecap} has been extended to the case of Sobolev functions. Now, we generalize this relation to the case of $BV$ functions, and, by doing so, we also immediately prove the inclusion of $BV(I)$ in $W^{s, 1}_{RL,a+}(I)$.

\begin{theorem} \label{result:BV_W_s1_RL} 
Let $u\in BV(I)$. Then, for all $s \in (0, 1)$, we have $u \in W^{s,1}_{RL,a+}(I)$ with
\begin{equation} \label{eq:repr_form_BV}
D^s_{a+}[u](x)= I^{1-s}_{a+}[Du](x)+\frac{1}{\Gamma\left(1 - s\right)} \frac{u(a+)}{(x-a)^{s}}.
\end{equation}
In particular, we have $BV(I) \hookrightarrow W^{s, 1}_{RL,a+}(I)$ for all $s \in (0, 1)$, with
\begin{equation} \label{eq:BV_W_s1_RL_emb}
\|u\|_{W^{s, 1}_{RL,a+}(I)} \le \max \left \{ 1 + \frac{(b - a)^{-s}}{\Gamma(2 - s)}, \frac{2 (b - a)^{1 - s}}{\Gamma(2 - s)} \right \} \|u\|_{BV(I)}.
\end{equation}
In addition, 
\begin{equation}
\label{eq:convinmeas}
D^s_{a+}[u] \Leb{1} \weakto Du + u(a+)\delta_a\quad\text{as}\quad s\to 1^-\quad\text{in}\quad \M(\overline{I}).
\end{equation}
\begin{proof}
By \cref{rem:frac_int_well_posed}, we obtain immediately that $I_{a+}^{1 - s}[u] \in L^{1}(I)$, since $u \in L^{1}(I)$. Let us now assume that $u \in AC(\overline{I})$. For all $x \in (a, b)$, formula \eqref{eq:RLecap} yields
\begin{equation*}
\frac{d}{dx} I_{a+}^{1 - s}[u](x) = I_{a+}^{1 - s}[u'](x) + \frac{1}{\Gamma\left(1 - s\right)} \frac{u(a)}{(x - a)^{s}}.
\end{equation*}
Now, let $u \in BV(I)$ and $\tilde{u}$ be its zero extension to $\R$ given by \eqref{eq:extension_zero_BV}. Let $\rho \in C^{\infty}_{c}((-1, 1))$ be a standard mollifier. It is well known that $\rho_{\eps} \ast \tilde{u} \in C^{\infty}(I) \cap BV(I)$, so that $\rho_{\eps} \ast \tilde{u} \in W^{1, 1}(I) \emb AC(\overline{I})$, in particular.
Then, for any $\phi \in C^{1}_c(I)$ we have
\begin{equation*}
\int_{a}^{b} I_{a+}^{1 - s} [\rho_{\eps} \ast \tilde{u}] \, \phi' \, dx = - \int_{a}^{b} \left (I_{a+}^{1 - s}[\rho_{\eps} \ast D \tilde{u}] + \frac{1}{\Gamma\left(1 - s\right)} \frac{(\rho_{\eps} \ast \tilde{u})(a)}{(x - a)^{s}} \right) \phi \, dx. 
\end{equation*}
By \eqref{eq:duality_fract_int}, we get
\begin{equation*}
\int_{a}^{b} I_{a+}^{1 - s}[\rho_{\eps} \ast D \tilde{u}] \, \phi \, dx = \int_{a}^{b} I_{b-}^{1 - s}[\phi] \, (\rho_{\eps} \ast D \tilde{u}) \, dx.
\end{equation*}
Then, since \cref{prop:cont_BLNT} applies also to $I^{1-s}_{b-}$ (thanks to Remark \ref{rem:reflection}), we have $I_{b-}^{1 - s}[\phi] \in C^{0}(\overline{I})$. Thanks to \eqref{eq:grad_zero_extension}, we obtain
\begin{equation*}
(\rho_{\eps} * D\tilde{u})(x) = (\rho_{\eps} * (Du \res I))(x) + u(a+) \rho_{\eps}(x - a) - u(b-) \rho_{\eps}(b - x).
\end{equation*}
Since $I_{b-}^{1 - s}[\phi] \in C^{0}(\overline{I})$, we have
\begin{align*}
\int_{a}^{b} I_{b-}^{1 - s}[\phi](x) u(a+) \rho_{\eps}(x - a) \, dx & \to \frac{u(a+)}{2}I_{b-}^{1 - s}[\phi](a) = \frac{u(a+)}{2\Gamma(1-s)} \int_{a}^{b}  \frac{\phi(x)}{(x - a)^s} dx
\end{align*}
and
\begin{align*}
\int_{a}^{b} I_{b-}^{1 - s}[\phi](x) u(b-) \rho_{\eps}(b - x) \, dx & \to \frac{u(b-)}{2} I_{b-}^{1 - s}[\phi](b) = 0,
\end{align*}
since $\int_{-1}^{1} \rho dx = 1$ and $\rho$ is even, and
\begin{equation*}
|I_{b-}^{1 - s}[\phi](x)| \le \frac{\|\phi\|_{L^{\infty}(I)}}{\Gamma(1-s)} \int_x^b \frac{1}{(t - x)^s} \, dt = \frac{\|\phi\|_{L^{\infty}(I)}}{\Gamma(2-s)} (b - x)^{1 - s} \to 0 \text{ as } x \to b-.
\end{equation*} 
Then, Fubini's theorem implies that
\begin{align*}
\int_{a}^{b} I_{b-}^{1 - s}[\phi](x) (\rho_{\eps} \ast (Du \res I))(x) \, dx & = \int_a^b \int_a^b  I_{b-}^{1 - s}[\phi](x) \rho_{\eps}(x - y) \, d Du(y) \, dx \\
& = \int_a^b \int_a^b  I_{b-}^{1 - s}[\phi](x) \rho_{\eps}(y - x) \, dx \, d Du(y) \\
& \to \int_{a}^{b} I_{b-}^{1 - s}[\phi](y) \, d Du(y), 
\end{align*}
where in the last step we employed well known convergence properties of mollifications of continuous functions.
All in all, we get
\begin{align*}
\int_{a}^{b} I_{b-}^{1 - s}[\phi] \, (\rho_{\eps} \ast D \tilde{u}) \, dx & = \int_{a}^{b} I_{b-}^{1 - s}[\phi] \, (\rho_{\eps} \ast (Du \res I)) \, dx + \\
& + \int_{a}^{b} I_{b-}^{1 - s}[\phi](x) \left ( u(a+) \rho_{\eps}(x - a) - u(b-) \rho_{\eps}(b - x) \right ) \, dx \\
& \to \int_{a}^{b} I_{b-}^{1 - s}[\phi] \, d Du + \frac{u(a+)}{2} I^{1-s}_{b-}[\phi](a) \\
& = \int_{a}^{b} \phi(x) \left (  I_{a+}^{1 - s}[Du](x) + \frac{1}{\Gamma\left(1 - s\right)} \frac{u(a+)}{2 (x - a)^{s}} \right ) \, dx.
\end{align*}
where the last equality follows from \eqref{eq:duality_fract_int_meas} and the definition of $I^{1-s}_{b-}$.
On the other hand, we also obtain
\begin{align*}
\int_{a}^{b} I_{a+}^{1 - s} [\rho_{\eps} \ast \tilde{u}] \, \phi' \, dx & = \int_{a}^{b} (\rho_{\eps} \ast \tilde{u}) \, I_{b-}^{1 - s}[\phi'] \, dx \\
& \to \int_{a}^{b} u \, I_{b-}^{1 - s}[\phi'] \, dx = \int_{a}^{b} I_{a+}^{1 - s}[u] \, \phi' \, dx,
\end{align*}
by \eqref{eq:duality_fract_int} and Lebesgue's dominated convergence theorem, since $I_{b-}^{1 - s}[\phi'] \in L^{1}(I)$ and 
\begin{equation*}
\|\rho_{\eps} \ast \tilde{u}\|_{L^{\infty}(I)} \le \|u\|_{L^{\infty}(I)} \le C_{a, b} \|u\|_{BV(I)},
\end{equation*}
by \eqref{eq:BV_bounded}.
Now, since $(\rho_{\eps} \ast \tilde{u})(a) \to \frac{u(a+)}{2}$ by \eqref{eq:extreme_moll_conv}, we get
\begin{align*}
\int_{a}^{b} I_{a+}^{1 - s}[u](x) \, \phi'(x) \, dx & = \lim_{\eps \to 0} \int_{a}^{b} I_{a+}^{1 - s} [\rho_{\eps} \ast \tilde{u}](x) \, \phi'(x) \, dx \\
& = - \lim_{\eps \to 0} \int_{a}^{b} I_{b-}^{1 - s}[\phi](x) \, (\rho_{\eps} \ast D \tilde{u})(x) + \frac{1}{\Gamma\left(1 - s\right)} \frac{(\rho_{\eps} \ast \tilde{u})(a)}{(x - a)^{s}} \, \phi(x) \, dx \\
& = - \int_{a}^{b} \phi(x) \left (  I_{a+}^{1 - s}[Du](x) + \frac{1}{\Gamma\left(1 - s\right)} \frac{u(a+)}{2 (x - a)^{s}} \right ) \, dx + \\
& - \frac{1}{\Gamma\left(1 - s\right)} \int_{a}^{b} \frac{u(a+)}{2(x - a)^{s}} \phi(x) \, dx \\
& = - \int_{a}^{b} \left ( I_{a+}^{1 - s}[Du](x) + \frac{1}{\Gamma\left(1 - s\right)} \frac{u(a+)}{(x - a)^{s}} \right ) \phi(x) \, dx,
\end{align*}
which yields \eqref{eq:repr_form_BV}. Thus, $D^s_{a+}u \in L^{1}(I)$, with
\begin{equation*}
\left \|D^s_{a+}u\right \|_{L^{1}(I)} \le \frac{(b - a)^{1 - s}}{\Gamma(2 - s)} (|D u|(I) + |u(a+)|)
\end{equation*}
by \eqref{eq:measure_frat_int}. Then, thanks to \eqref{eq:BV_bounded} and \eqref{eq:lim_ab_bound}, we get
\begin{align*}
\|u\|_{W^{s, 1}_{RL,a+}(I)} & = \|u\|_{L^{1}(I)} + \left \|D^{s}_{a+}u\right \|_{L^{1}(I)} \le  \|u\|_{L^{1}(I)} +  \frac{(b - a)^{1 - s}}{\Gamma(2 - s)} \left (|D u|(I) + |u(a+)| \right ) \\
& \le \|u\|_{L^{1}(I)} +  \frac{(b - a)^{1 - s}}{\Gamma(2 - s)} \left (|D u|(I) + \frac{1}{b - a} \|u\|_{L^{1}(I)} + |D u|(I) \right )  \\
& = \left ( 1 + \frac{(b - a)^{- s}}{\Gamma(2 - s)} \right ) \|u\|_{L^{1}(I)} + \frac{2 (b - a)^{1 - s}}{\Gamma(2 - s)} |D u|(I),
\end{align*}
which easily implies \eqref{eq:BV_W_s1_RL_emb} and the continuity of the embedding $BV(I) \hookrightarrow W^{s, 1}_{RL,a+}(I)$.

To prove the second part of the claim, we choose $\phi \in C^{1}(\overline{I})$ and exploit \eqref{eq:duality_fract_int_meas} and \eqref{eq:repr_form_BV} in order to obtain
\begin{align*}
\int_a^b D^s_{a+}[u](x)\phi(x)dx&=\int_a^b I^{1-s}_{a+}[Du](x)\phi(x)dx+\frac{u(a+)}{\Gamma(1-s)}\int_a^b\frac{\phi(x)}{(x-a)^s}dx \\
&=\int_a^b I^{1-s}_{b-}[\phi](x)dDu(x)+ \\
& + \frac{u(a+)}{\Gamma(2-s)}\left(\phi(b)(b-a)^{1-s}-\int_a^b\phi'(x)(x-a)^{1-s}dx\right).
\end{align*}
Therefore, by \cref{lem:fractional_integral_s_0} and Lebesgue's dominated convergence theorem, we get
$$
\lim_{s\to 1^-}\int_a^b D^s_{a+}[u](x)\phi(x)dx=\int_a^b\phi(x)dDu(x)+u(a+)\phi(a).
$$
Then, the claim plainly follows by the density of $C^1(\overline{I})$ in $C(\overline{I})$ with respect to the supremum norm.
\end{proof}
\end{theorem}

	\begin{lemma}\label{lem:density}
	Let $s\in(0,1)$. If $u\in W^{s,1}(I)$, then $D^{s}_{a+}[u]$ is well defined, belongs to $L^{1}(I)$ and $D^{s}_{a+}[u](x) = {}^M D^s_{a+}[u](x)$ for $\Leb{1}$-a.e. $x\in I$.
	\end{lemma}
	\begin{proof}
		If $u\in W^{s,1}(I)\cap AC(\overline{I})$, the computations already done in \eqref{eq:Marchaud_equiv} hold true.
		
		Otherwise, if $u\in W^{s,1}(I)$, we exploit the density of $C_c^1(I)$ in $W^{s,1}(I)$ (Remark \ref{rem:densitàdellecc1}), which means that there exists a sequence $u_n$ in $C_c^1(I)$ such that $\left\|u_n-u\right\|_{W^{s,1}(I)}\to 0$ as $n \to + \infty$. Now, we prove that
\begin{equation} \label{eq:Marchaud_repr_smooth_seq}
		D^s_{a+}[u_n](x)=\frac{1}{\Gamma(1-s)}\frac{u_n(x)}{(x-a)^s}+\frac{s}{\Gamma(1-s)}\int_a^x\frac{u_n(x)-u_n(t)}{(x-t)^{s+1}}dt
\end{equation}
		converges in $L^{1}(I)$ and, up to a subsequence, pointwise $\Leb{1}$-a.e. in $I$ to $D^s_{a+}[u](x)$.
		
		For the second term in the right hand side of \eqref{eq:Marchaud_repr_smooth_seq}, we proceed as follows: we set
		$$
		f_n(x):=\int_a^x\frac{u_n(x)-u_n(t)-u(x)+u(t)}{(x-t)^{s+1}}dt.
		$$
		The sequence $f_n$ converges to 0 in $L^1(I)$. Indeed
		$$
		\int_a^b|f_n(x)|dx\leq[u_n-u]_{W^{s,1}(I)}\leq\left\|u_n-u\right\|_{W^{s,1}(I)}\to 0\quad\text{as}\quad n \to + \infty.
		$$
		Therefore, up to a subsequence, $f_n$ converges pointwise $\Leb{1}$-a.e. to 0 in $I$, so that
		$$
		\lim_{n \to + \infty}\int_a^x\frac{u_n(x)-u_n(t)}{(x-t)^{s+1}} \, dt =\int_a^x\frac{u(x)-u(t)}{(x-t)^{s+1}} \, dt
		$$
		for $\Leb{1}$-a.e. $x \in I$.
		Conversely, for the first term in the right hand side of \eqref{eq:Marchaud_repr_smooth_seq}, up to a subsequence, we have convergence $\Leb{1}$-a.e. in $I$ thanks to the convergence of $u_n$ to $u$ in $W^{s,1}(I)$ and hence in $L^1(I)$, which implies pointwise convergence $\Leb{1}$-a.e., up to a subsequence. 
		
		For the $L^1$ convergence, we argue as follows: 	
		employing the fractional Hardy inequality, \cref{lem:hardy}, with $p=1$, we get
		$$
		\int_a^b\frac{|u_n(x)-u(x)|}{(x-a)^s}dx\leq\int_a^b\frac{|u_n(x)-u(x)|}{|\delta_I(x)|^s}dx\leq C\left\|u_n-u\right\|_{W^{s,1}(I)}\to 0\quad\text{as}\quad n \to + \infty.
		$$
		To conclude, we notice that, for any $\phi\in C^1_c(I)$ it holds that
		\begin{align*}
		\int_a^b {}^M D^s_{a+}[u_n](x)\phi(x)dx & =\int_a^b D^s_{a+}[u_n](x)\phi(x)dx \\ & =-\int_a^b I^{1-s}_{a+}[u_n](x)\phi'(x)dx \rightarrow-\int_a^b I^{1-s}_{a+}[u](x)\phi'(x)dx,
		\end{align*}
		since $u_n\rightarrow u$ in $L^1(I)$ and $I^{1-s}_{a+}$ is continuous from $L^1(I)$ to $L^1(I)$.
		On the other hand, we have just proved that ${}^M D^s_{a+}[u_n]\to {}^M D^s_{a+}[u]$ in $L^1(I)$, and so we conclude
		$$
		\int_a^b {}^M D^s_{a+}[u](x)\phi(x)dx=-\int_a^b I^{1-s}_{a+}[u](x)\phi'(x)dx,
		$$
		and this implies $u\in W^{s,1}_{RL,a+}(I)$ with ${}^M D^s_{a+}[u](x)=D^s_{a+}[u](x)$ for a.e. $x\in I$.
	\end{proof}

\begin{remark}
	We notice that H\"older's inequality cannot be exploited in order to estimate the term $$\int_a^b\frac{|u_n(x)-u(x)|}{(x-a)^s}dx$$ in the proof of \cref{lem:density}, so that we need to employ the fractional Hardy inequality of \cref{lem:hardy}. Indeed, since $u_n-u\in W^{s,1}(I)$, the fractional Sobolev embedding Theorem (see {\it e.g.} \cite[Theorem 6.7.]{MR2944369}) implies that  $u_n-u\in L^q(I)$ for any $q\in\left[1,\frac{1}{1-s}\right]$.  
	
	Therefore, we get
	$$
	\int_a^b\frac{|u_n(x)-u(x)|}{(x-a)^s}dx\leq\left(\int_a^b|u_n-u|^q dx\right)^{1/q}\left(\int_a^b\frac{dx}{(x-a)^{sq'}}\right)^{1/q'}.
	$$
	Now, $q \le \frac{1}{1-s}$ implies $sq' \ge 1$, and so
	$$
	\int_a^b\frac{dx}{(x-a)^{sq'}} = + \infty,
	$$
	and thus this estimate is not useful.
\end{remark}

\begin{proposition}
\label{prop:propopop}
For all $s\in(0,1)$ the embedding
$W^{s,1}(I)\hookrightarrow W^{s,1}_{RL,a+}(I)$ is continuous.
\end{proposition}
\begin{proof}
Since $u\in L^{1}(I)$, in particular, we have $I^{1-s}_{a+}[u]\in L^1(I)$ by \cref{rem:frac_int_well_posed}.

Thanks to Lemma \ref{lem:density}, we have that the left Riemann-Liouville fractional derivative of $u$ coincides with the Marchaud one, and so
\begin{equation*}
D^s_{a+}[u](x)=\frac{1}{\Gamma(1-s)}\frac{u(x)}{(x-a)^s}+\frac{s}{\Gamma(1-s)}\int_a^x\frac{u(x)-u(t)}{(x-t)^{s+1}}dt \text{ for } \Leb{1}\text{-a.e } x \in I.
\end{equation*}
For the second term on the right hand side, it holds that
\begin{equation}
\label{eq:loplm}
\int_a^b \left|\int_a^x\frac{u(x)-u(t)}{(x-t)^{s+1}}dt\right| dx \leq [u]_{W^{s,1}(I)}.
\end{equation}
While for the first term, using Lemma \ref{lem:hardy} with $p=1$, we have
\begin{equation}
\label{eq:llllll}
\int_a^b\frac{|u(x)|}{(x-a)^s}dx \leq\int_a^b\frac{|u(x)|}{|\delta_I(x)|^s}dx\leq C\left\|u\right\|_{W^{s,1}(I)}
\end{equation}
for some $C = C(s, a, b) > 0$.

All in all, using \eqref{eq:measure_frat_int}, \eqref{eq:loplm} and \eqref{eq:llllll}, we obtain that there exists a positive constant $C=C(s, a, b)$ such that
\begin{equation*}
\left\|u\right\|_{W^{s,1}_{RL,a+}(I)}\leq C\left\|u\right\|_{W^{s,1}(I)}.
\end{equation*}
\end{proof}

We notice that, thanks to the continuous embedding $BV(I) \hookrightarrow W^{s, 1}(I)$ (see for instance \cite[Proposition 1.2.1]{TesiLuca}), \cref{prop:propopop} actually implies the continuous embedding $BV(I) \hookrightarrow W^{s, 1}_{RL, a+}(I)$, also given by \cref{result:BV_W_s1_RL}. However, the proofs of these two results exploit different techniques, both of them interesting in their own way.

\begin{remark} \label{rem:dubious}
	We notice that Proposition \ref{prop:propopop} does not hold for unbounded intervals. Indeed, the function $u(x):=\frac{1}{x^2}$ belongs to $W^{1,1}((1,+\infty))$, therefore $u\in W^{1/2,1}((1,+\infty))$, but we have 
	$$
	I^{1/2}_{1+}[u](x)=\frac{1}{\sqrt{\pi}}\left(\frac{\log(x)+2\log\left(1+\sqrt{\frac{x-1}{x}}\right)}{2x^{3/2}}+\frac{\sqrt{x-1}}{x}\right)\notin L^1((1,+\infty)).
	$$
This example shows also that the continuity of the fractional integral in $L^p$ for $1\le p \le 2$ fails for unbounded intervals.
\end{remark}

Actually, we can prove that the inclusion of \cref{prop:propopop} is strict.

\begin{proposition} \label{prop:strict_embedding}
For all $s \in (0, 1)$ the space $W^{s,1}_{RL,a+}(I)$ strictly contains $W^{s,1}(I)$, and so it strictly contains $BV(I)$.
\end{proposition}
	\begin{proof}
We claim that the function $$u(x):=\frac{(x - a)^{s-1}}{\Gamma(s)}$$ belongs to $W^{s,1}_{RL,a+}(I)\setminus W^{s,1}(I)$.
By Example \ref{ex:critical_power}, we know that $u \in L^1(I)$ and that $I^{1-s}_{a+}[u](x)=1$ for all $x \in I$, so that $I^{1-s}_{a+}[u]\in W^{1,1}(I)$, which implies that $u \in W^{s, 1}_{RL, a+}(I)$, by definition.
Then, we need to prove that the Gagliardo-Slobodeckij seminorm of $u$ is infinite. We see that
\begin{align*}
\Gamma(s) [u]_{W^{s,1}(I)} & := \int_{a}^{b} \int_{a}^{b}\frac{|(x - a)^{s-1} - (y - a)^{s-1}|}{|x-y|^{s+1}}\, dx \, dy = \left [ \begin{matrix}x = a + (b - a) x\\y = a + (b - a)y \end{matrix} \right ] \\
& = \int_{0}^{1} \int_{0}^{1}\frac{|x^{s-1} - y^{s-1}|}{|x-y|^{s+1}}\, dx \, dy = [x = y z] = \int_{0}^{1} \int_{0}^{\frac{1}{y}} \frac{|z^{s-1} - 1| y^{s-1}}{|z - 1|^{s + 1} y^{s + 1}} y \, d z \, dy \\
& = \int_{0}^{\infty} \int_{0}^{\min\left\{1, \frac{1}{z}\right\}} \frac{|z^{s-1} - 1|}{|z - 1|^{s + 1}y} \, d y \, dz = + \infty
\end{align*}
since $1/y \notin L^{1}((0, \delta))$, for any $\delta > 0$. Finally, we recall that $W^{s, 1}(I)$ contains $BV(I)$ and this ends the proof.
\end{proof}

\begin{remark}
In particular, we see that, for all $s \in (0, 1)$, all the inclusions $$BV(I) \hookrightarrow W^{s, 1}(I)  \hookrightarrow W^{s, 1}_{RL, a+}(I)$$ are strict: indeed, if we consider $$v_{\sigma}(x) := (x - a)^{\sigma - 1}$$
for some $\sigma \in (s, 1)$, then, arguing as in the proof of Proposition \ref{prop:strict_embedding}, it is easy to see that $v_{\sigma} \in W^{s, 1}(I) \setminus BV(I)$.
\end{remark}

	Now, we recall the fact that $W^{s,1}(I)$ is a real interpolation space; namely, $$W^{s,1}(I)=(L^1(I),W^{1,1}(I))_{s,1}.$$ 
	More generally, we have that, for $s\in(0,1)$, $1\le p\le\infty$, $1\le q\le \infty$, it holds that $$(L^p(I),W^{1,p}(I))_{s,q}=B^s_{p,q}(I),$$ where $B^s_{p,q}(I)$ denotes the Besov space; in particular, if $p=q$ we have $B^s_{p,p}(I)=W^{s,p}(I)$.
	In light of this observation, we can extend Proposition \ref{prop:propopop} in the following way.
	
\begin{corollary}
	\label{cor:besov}
	For $0<s<r<1$, $1\le p\le \infty$ and $1\le q \le \infty$, we have that the embedding $B^r_{p,q}(I)\hookrightarrow W^{s,1}_{RL,a+}(I)$ is continuous.
	\end{corollary}
	\begin{proof}
		Using \cite[Proposition 1.4]{MR3753604}, we have that, for $r,s,p,q$ as given in our claim, the Besov space $B^r_{p,q}(I)$ is continuously embedded in $B^s_{1,1}(I)=W^{s,1}(I)$, which is continuously embedded in $W^{s,1}_{RL,a+}(I)$ thanks to Proposition \ref{prop:propopop}, and this proves the claim. \end{proof}

\begin{remark}
	\label{rem:mironsic}
	Unfortunately, Corollary \ref{cor:besov} does not cover the case $r=s$ for any choice of $p$ and $q$. Therefore, in the particular case $p=q$ we are unable to conclude that Proposition \ref{prop:propopop} extends to $W^{s,p}(I)$ for any $p>1$. Indeed, in \cite{MR3357858} the authors prove that for $1\le q<p\le \infty$ and $s>0$, $s \notin \N$, we have $W^{s,p}(I)\not\subseteq W^{s,q}(I)$.
\end{remark}

Now, we extend Proposition \ref{prop:propopop} to the case $p>1$.

\begin{proposition}
	\label{prop:newprop}
	Let $p \in [1, \infty)$ and $s\in \left ( 0,\frac{1}{p} \right )$. Then it holds that
	\begin{equation}
	W^{r,p}(I)\hookrightarrow W^{s,p}_{RL,a+}(I),
	\end{equation}
	for all $r\in \left [s+\frac{1}{p'},1 \right )$, where $p'$ denotes the H\"older conjugate of $p$.
\end{proposition}
\begin{proof}
First of all, setting $\sigma:=s+\frac{1}{p'}$, we notice that, since $s<\frac{1}{p}$, we have $s<\sigma<1$.
	Moreover, since $W^{r,p}(I)\hookrightarrow W^{\sigma,p}(I)$ for any $r\in(\sigma,1)$ (see {\it e.g.} \cite[Proposition 2.1.]{MR2944369}), we reduce ourselves to prove that $W^{\sigma,p}(I)\emb W^{s,p}_{RL,a+}(I)$.
	
Now, the case $p=1$ follows immediately from Corollary \ref{cor:besov} with $q=p$.
	
Then, let $p > 1$ and $u\in W^{\sigma,p}(I)$. By \cite[Proposition 1.4]{MR3753604} we have $u\in W^{s,1}(I)$, and so, using Proposition \ref{prop:propopop} and Lemma \ref{lem:density}, it follows that $u\in W^{s,1}_{RL,a+}(I)$ with $D^s_{a+}[u](x)={}^M D^s_{a+}[u](x)$ for $\Leb{1}$-a.e. $x\in I$.
	
To conclude, it is sufficient to estimate the $L^p(I)$-norm of 
	\begin{equation}
	\label{eq:raitendsaid}
	D^s_{a+}[u](x) = {}^M D^s_{a+}[u](x) =\frac{1}{\Gamma(1-s)}\frac{u(x)}{(x-a)^s}+\frac{s}{\Gamma(1-s)}\int_a^x\frac{u(x)-u(t)}{(x-t)^{s+1}}dt.
	\end{equation}
	Since $sp<1$ and $\sigma>s$, using Lemma \ref{lem:hardy} and \cite[Proposition 2.1]{MR2944369}, for the first term in the right-hand side of \eqref{eq:raitendsaid} it holds that
	\begin{equation}
	\label{eq:usinghardy}
	\int_a^b\frac{|u(x)|^p}{(x-a)^{sp}}dx\le C\left\|u\right\|^p_{W^{s,p}(I)}\le C\left\|u\right\|^p_{W^{\sigma,p}(I)},
	\end{equation}
	for some constant $C>0$.
	
As for the second term in the right-hand side of \eqref{eq:raitendsaid}, we exploit H\"older's inequality to get 
	\begin{align*}
	\int_a^b \left|\int_a^x\frac{u(x)-u(t)}{(x-t)^{s+1}} \, dt\right|^p dx & \leq\int_a^b \int_a^x\frac{|u(x)-u(t)|^p}{(x-t)^{sp+p}}\, dt (x - a)^{p - 1} \, dx \\
	& \le (b - a)^{p - 1} \int_a^b \int_a^b\frac{|u(x)-u(t)|^p}{|x-t|^{sp+p}} \, dt \, dx.
	\end{align*}
Therefore, since $sp+p=\sigma p+1$, we have
\begin{equation} \label{eq:lasttermraitend}
\int_a^b \left|\int_a^x\frac{u(x)-u(t)}{(x-t)^{s+1}} \, dt\right|^p dx \le (b - a)^{p - 1} [u]^p_{W^{\sigma,p}(I)}.
\end{equation}
By combining \eqref{eq:usinghardy} and \eqref{eq:lasttermraitend}, we immediately obtain 
\begin{equation*}
	\left\|D^s_{a+}[u]\right\|_{L^p(I)}\le\max\{(b - a)^{\frac{1}{p'}},C\}\left\|u\right\|_{W^{\sigma,p}(I)},
\end{equation*}
	and this concludes the proof.
	\end{proof}

\begin{remark}
It is interesting to notice that some inequalities of Poincaré type hold true in the fractional context, for which we refer for instance to \cite[Chapter 17]{MR2513750}. However, in general it is not possible to retrieve the classical Poincaré inequality by estimating the $L^{p}$ norm of the difference between $u$ and its average with the $L^{p}$ norm of its Riemann-Liouville derivative. Indeed, let us consider the function $u(x):=(x-a)^{s-1}$ for some $s\in(0,1)$. By Example \ref{ex:critical_power}, we have $u\in L^p(I)$ for all $1 \le p<\frac{1}{1-s}$, and
\begin{equation*}
I_{a+}^{1 - s}[u](x) = \Gamma(s) \ \text{ and } \ D^{s}_{a+}[u](x) = 0 \text{ for all } x \in I.
\end{equation*} 
Therefore, $u\in W^{s,p}_{RL,a+}(I)$ for all $p \in \left [1, \frac{1}{1 - s} \right )$. Thus, being $u$ not constant, we cannot hope for any sort of Poincaré inequality of the type
\begin{equation*}
\int_{a}^{b} |u - u_{I}|^{p} \, dx \le C \int_{a}^{b} |D^{s}_{a+}[u]|^p \, dx,
\end{equation*}
where 
	$$
	u_I:=\frac{1}{b - a}\int_a^b u(x) \, dx = \frac{1}{s} (b - a)^{s - 1}.
	$$
\end{remark}

\smallskip

\section{The space $BV^s_{RL,a+}(I)$} \label{sec:BV_s}

In analogy with the previous definition of left Riemann-Liouville fractional Sobolev spaces, we introduce now the natural extension to the $BV$ framework.

\begin{definition} Let $s\in(0,1)$. We define the space of {\em functions with left Riemann-Liouville fractional bounded variation} as
\begin{equation*}
BV^{s}_{RL,a+}\left(I\right):=\left\{u\in L^1\left(I\right) : \ I^{1-s}_{a+}\left[u\right]\in BV(I) \right\}.
\end{equation*}
\end{definition}

From the definition, it follows immediately that $u$ belongs to $BV^{s}_{RL,a+}\left(I\right)$ if and only if there exists a measure $\mu^{s} \in\mathcal{M}(I)$ satisfying
\begin{equation*}
\int_a^b I^{1-s}_{a+}[u](x) \, \phi'(x) \, dx = - \int_a^b \phi(x) \, d\mu^{s}(x)
\end{equation*}
for any $\phi\in C_c^1(I)$, and we call $\mathcal{D}^{s}_{a+}[u] = D I^{1-s}_{a+}[u]:=\mu^{s}$ the left Riemann-Liouville distributional $s$-fractional derivative. In order to avoid any ambiguity, we denote with $D^s_{a+}[u]$ the density of the absolutely continuous part of $\mathcal{D}^s_{a+}[u]$ with respect to the Lebesgue measure $\Leb{1}$.

It is not difficult to see that the space $BV^{s}_{RL,a+}(I)$, endowed with the norm
$$
\left\|u\right\|_{BV^s_{RL,a+}(I)}:=\left\|u\right\|_{L^{1}(I)} + \left\|I^{1-s}_{a+}[u]\right\|_{BV(I)},
$$
is a Banach space.

Arguing analogously as in Lemma \ref{result:intbp}, we derive a duality relation between the left Riemann-Liouville weak $s$-fractional derivative and the right Caputo $s$-fractional derivative.

\begin{corollary} \label{cor:Caputo_der_dual}
	A function $u\in L^1(I)$ belongs to $BV^s_{RL,a+}(I)$ if and only if there exists $\mu^{s} \in\mathcal{M}(I)$ such that 
	$$
	\int_a^b u(x) \, {}^C D_{b-}^s[\phi](x) \, dx=\int_a^b\phi(x) \, d\mu^{s}(x)
	$$
	for every $\phi\in C_c^1(I)$. In that case, we have $\mathcal{D}^{s}_{a+}[u]=\mu^{s}$.
\end{corollary}

It is natural to ask what we can say about the decomposition of the measure $\mathcal{D}^{s}_{a+}[u]$ for a function $u\in BV^s_{RL,a+}(I)$. We start with the following result.

\begin{proposition} \label{prop:BV_s_W_s_a}
Let $u \in L^{1}(I)$ and $s \in (0, 1)$. Then $u \in W^{s, 1}_{RL, a+}(I)$ if and only if $u \in BV^{s}_{RL, a+}(I)$ and $\mathcal{D}^{s}_{a+}[u] \ll \Leb{1}$. In particular, the embedding $W^{s, 1}_{RL, a+}(I) \emb BV^{s}_{RL, a+}(I)$ is continuous, and $\mathcal{D}^{s}_{a+}[u] = D^{s}_{a+}[u] \Leb{1}$ for any $u \in W^{s, 1}_{RL, a+}(I)$. In addition, if $u\in BV(I)$, then 
\begin{equation*}
\mathcal{D}^s_{a+}[u]= \left ( I^{1-s}_{a+}[Du] + \frac{1}{\Gamma\left(1 - s\right)} \frac{u(a+)}{(\cdot -a)^{s}} \right ) \Leb{1}.
\end{equation*}
\end{proposition}
	\begin{proof}
If $u \in W^{s, 1}_{RL, a+}(I)$, Lemma \ref{result:intbp} implies that 
	$$
	\int_a^b u(x) \, {}^C D_{b-}^s[\phi](x) \, dx=\int_a^b\phi(x) D^{s}_{a+}[u](x) \, dx 
	$$
	for every $\phi\in C_c^1(I)$. Thus, Corollary \ref{cor:Caputo_der_dual} implies that $u \in BV^{s}_{RL, a+}(I)$ with $$\mathcal{D}^{s}_{a+}[u] = D^{s}_{a+}[u] \Leb{1}.$$  
	This immediately implies that the embedding $W^{s,1}_{RL,a+}(I) \emb BV^{s}_{RL, a+}(I)$ is continuous.
	On the other hand, if $u \in BV^{s}_{RL, a+}(I)$ and $\mathcal{D}^{s}_{a+}[u] \ll \Leb{1}$, then $I^{1 - s}_{a+}[u]$ belongs to $W^{1, 1}(I)$, and so $u \in W^{s, 1}_{RL, a+}(I)$.
By Theorem \ref{result:BV_W_s1_RL}, if $u\in BV(I)$, then $u\in W^{s,1}_{RL,a+}(I)$, and the representation formula is a trivial consequence of \eqref{eq:repr_form_BV}.
\end{proof}

In the spirit of Lemma \ref{lem:invinteder}, we can obtain a version of the Fundamental Theorem of Calculus for functions in $BV^{s}_{RL, a+}(I)$.

\begin{lemma}\label{lem:invinteder_BV}
Let $s\in(0,1)$ and $u\in BV^{s}_{RL,a+}(I)$. Then, for $\Leb{1}$-a.e. $x\in I$, we also have 
\begin{equation}\label{eq:invinteder_1_BV}
		u(x) = D^s_{a+}[I^s_{a+}[u]](x) =  I^s_{a+}[\mathcal{D}^s_{a+}[u]](x)+\frac{I^{1-s}_{a+}[u](a+)}{\Gamma(s)}(x-a)^{s-1}.
\end{equation}
In addition, if $u \in BV^{s}_{RL,a+}(I) \cap I^s_{a+}(L^1(I))$, then $u \in W^{s, 1}_{RL, a+}(I) \cap I^{s}_{a+}(L^{1}(I)), I^{1 - s}_{a+}[u](a+) = 0$ and \eqref{eq:invinteder_3} holds.
\end{lemma}
\begin{proof}
The first equality in \eqref{eq:invinteder_1_BV} follows immediately from \eqref{eq:invinteder_1}. The second one can be proved as \eqref{eq:invinteder_2}. Indeed, if $u \in BV^{s}_{RL, a+}(I)$, then $I^{1-s}_{a+}[u] \in BV(I)$ with weak derivative $\mathcal{D}^{s}_{a+}[u]$. Therefore, by \cite[Theorem 3.28]{AFP}, for $\Leb{1}$-a.e. $x \in I$, we get
\begin{equation*}
			\begin{split}
			I^{1-s}_{a+}[u](x)&=\int_a^x \, d \mathcal{D}^s_{a+}[u](t) + I^{1-s}_{a+}[u](a+) \\
			&=I^{1-s}_{a+}[I^s_{a+}[\mathcal{D}^s_{a+}[u]]](x)+I^{1-s}_{a+}\left[\frac{I^{1-s}_{a+}[u](a+)}{\Gamma(s)}(\cdot-a)^{s-1}\right](x)
			\end{split}
			\end{equation*}
by \eqref{eq:frac_int_critical_power}.
We notice that $\mathcal{D}^{s}_{a+}[u] \in \mathcal{M}(I)$, and so, by Proposition \ref{result:measure_fract_int}, $I^{s}_{a+}[\mathcal{D}^{s}_{a+}[u]] \in L^{1}(I)$. Thus, it is enough to apply $D^{1-s}_{a+}$ to both sides of the equation and use \eqref{eq:invinteder_1} to obtain \eqref{eq:invinteder_1_BV}. 
Finally, if $u\in BV^{s}_{RL,a+}(I) \cap I^s_{a+}(L^1(I))$, then, by Lemma \ref{lem:image_fract_int_Sob} with $p=1$, we have $u \in W^{s, 1}_{RL, a+}(I), I^{1-s}_{a+}[u](a+)=0$, and so it satisfies the hypotheses for \eqref{eq:invinteder_3}. This ends the proof.
	\end{proof}

\begin{remark}
	\label{rem:klopp}
As a trivial consequence of Proposition \ref{prop:BV_s_W_s_a}, we see that, if $u\in W^{s,p}_{RL,a+}(I)$ for some $p\ge 1$, and $s\in(0,1)$, then $u\in BV^s_{RL,a+}(I)$, and $\mathcal{D}^s_{a+}[u] \ll \Leb{1}$, with density given by the left Riemann-Liouville weak $s$-fractional derivative.
\end{remark}

We show that the inclusion of $W^{s, 1}_{RL, a+}(I)$ into $BV^s_{RL,a+}(I)$ is strict, by constructing an example of a function $u$ such that the measure $\mathcal{D}^s_{a+}[u]$ is not absolutely continuous with respect $\Leb{1}$.

\begin{example}[$BV^s_{RL,a+}(I)\setminus W^{s, 1}_{RL, a+}(I) \neq \emptyset$] \label{rem:BV^s_RL_BV_inclusion}
Let $s\in(0,1)$, $J=(c,d)$ with $c,d\in\mathbb{R}$ such that $a<c<d<b$. We define the following function
\begin{equation*}
u(x):=
\begin{cases}
0 & \text{ if } \ a < x \le c, \\
\displaystyle \frac{(x-c)^{s-1}}{\Gamma(s)} & \text{ if } \ c<x\le d, \\
0 & \text{ if } \ d<x < b.
\end{cases}
\end{equation*}
Now, we compute $I^{1-s}_{a+}[u](x)$. Clearly, when $x\in (a,c)$, $I^{1-s}_{a+}[u](x)=0$. On the other hand, for $x\in J$ by \eqref{eq:frac_int_critical_power} we obtain
\begin{equation*}
I^{1-s}_{a+}[u](x)=\frac{1}{\Gamma(s)\Gamma(1-s)}\int_c^x (t-c)^{s-1}(x-t)^{-s} \, dt=1.
\end{equation*}
Therefore, for any $x\in I$, we have
\begin{equation*}
I^{1-s}_{a+}[u](x)=\begin{cases}
0 & \text{ if } \ x\in(a,c] \\
1 & \text{ if } \ x\in(c,d] \\
\displaystyle\frac{1}{\Gamma(s)\Gamma(1-s)}\int_c^d (t-c)^{s-1}(x-t)^{-s}dt & \text{ if } \ x\in(d,b).
\end{cases}
\end{equation*}
In other words, we have
\begin{equation*}
I^{1-s}_{a+}[u](x) = \chi_J(x)+f(x)\chi_{(d,b)}(x) \ \text{ for } \Leb{1}\text{-a.e. } x \in I,
\end{equation*}
where 
\begin{equation*}
f(x):=\frac{1}{\Gamma(s)\Gamma(1-s)}\int_c^d (t-c)^{s-1}(x-t)^{-s}dt.
\end{equation*}
It is not difficult to see that 
\begin{equation*}
f(d)=\frac{1}{\Gamma(s)\Gamma(1-s)}\beta(s,1-s)=1,
\end{equation*}
and $f \in C([d,b))\cap C^\infty((d,b))\cap W^{1, 1}((d,b))$. Hence, we deduce that
\begin{equation*}
\mathcal{D}^s_{a+}[u]=\delta_c-\delta_d +f' \chi_{(d,b)}\Leb{1}+f(d)\delta_d=\delta_c+f' \Leb{1} \res (d, b)
\end{equation*} 
so that $I^{1-s}_{a+}[u]\in BV(I) \setminus W^{1, 1}(I)$. Thus, $u\in BV^s_{RL,a+}(I)\setminus W^{s, 1}_{RL, a+}(I)$.
\end{example}

\begin{proposition}
Let $s \in (0, 1)$. Then, the inclusion $W^{s, 1}_{RL, a+}(I) \emb BV^s_{RL,a+}(I)$ is strict. 
\end{proposition}
\begin{proof}
It is an immediate consequence of Proposition \ref{prop:BV_s_W_s_a} and Example \ref{rem:BV^s_RL_BV_inclusion}.
\end{proof}

Thanks to Example \ref{rem:BV^s_RL_BV_inclusion}, we see that, given $u\in BV^s_{RL,a+}(I)\setminus W^{s, 1}_{RL, a+}(I)$, the measure $\mathcal{D}^s_{a+}[u]$ can have a jump part. It is natural to ask whether it admits also a Cantor part, in general. To this purpose, we exhibit an example of $u\in BV^s_{RL,a+}(I)$ such that $I^{1-s}_{a+}[u]\in BV(I)\setminus SBV(I)$, where $SBV(I)$ is the space of {\em special functions of bounded variation}, for which $(D u)_{c} = 0$.

\begin{example}
	Consider the classical ternary Cantor function $C(x)$, and let $I=(0,1)$. It is well known that $C\in C^{0,\alpha_C}(\overline{I})\cap BV(I)$, where $\alpha_C:=\log_32$, and $DC$ is a singular measure without atoms which means that $DC=(DC)_c$; in particular, up to a multiplicative constant, $DC=\mathcal{H}^{\alpha_C}$, see {\it e.g.\ \cite[Example 3.34]{AFP}}.
	
	Now, since $C(0)=0$, we can use Proposition \ref{prop:holcontinuity} to conclude that $C$ is representable as the $(1-s)$-fractional integral of a function in $C_0^{0,\alpha_C+s-1}(\overline{I})$, provided $s\in(1-\alpha_C,1)$. This implies that there exists $u\in C_0^{0,\alpha_C+s-1}(\overline{I})$ such that $I^{1-s}_{0+}[u](x)=C(x)$, and so $u\in BV^s_{RL,0+}(I)$, with $\mathcal{D}^s_{0+}[u]=DC=(DC)_c$.
\end{example}

\section{Action of the fractional integral on Sobolev functions} \label{sec:Sob_act}

Now we analyze regularizing properties of the fractional integral when it acts on functions in the Sobolev space $W^{1,p}(I)$ for some $p\ge 1$.
We start with the following statement.

\begin{proposition}
	\label{prop:regsobolev}
	Let $1\le p<\infty$ and $s\in(0,1)$ such that $sp<1$. Then $I^{1-s}_{a+}$ is a continuous operator from $W^{1,p}(I)$ into $W^{1,p}(I)$, with
	\begin{equation} \label{eq:continuity_I_a_s_Sobolev}
	\|I^{1 - s}_{a+}[u]\|_{W^{1, p}(I)} \le \frac{(b-a)^{1-s}}{\Gamma(1-s)} \left (\frac{1}{1 - s} + \max\left\{1,\frac{1}{b - a}\right\} \frac{1}{(1 - sp)^{\frac{1}{p}}} \right ) \|u\|_{W^{1, p}(I)}
	\end{equation}
	for all $u \in W^{1, p}(I)$.
In addition, $I^{1-s}_{a+}$ is a continuous operator from $W^{1,p}_{a}(I)$ into $W^{1,p}(I)$ for all $s \in (0, 1)$ and $p \in [1, \infty]$, with
\begin{equation} \label{eq:continuity_I_a_s_Sobolev_1}
	\|I^{1 - s}_{a+}[u]\|_{W^{1, p}(I)} \le \frac{(b-a)^{1-s}}{\Gamma(2-s)} \|u\|_{W^{1, p}(I)} \text{ for all } u \in W^{1, p}_{a}(I).
	\end{equation}
	\end{proposition}
	\begin{proof}
		Thanks to Proposition \ref{prop:cont_BLNT} and Remark \ref{rem:constant_Riesz_Thorin}, $I^{1-s}_{a+}[u]\in L^p(I)$, and
		\begin{equation}
		\label{eq:continLp}
		\left\|I^{1-s}_{a+}[u]\right\|_{L^p(I)}\leq \frac{(b-a)^{1-s}}{\Gamma(2-s)}\left\|u\right\|_{L^p(I)}.
		\end{equation} 		
		Now, we prove that $D^{s}_{a+}[u]\in L^p(I)$. Recalling Remark \ref{rem:Sobolev_AC_emb}, we have $u\in W^{1,p}(I) \emb BV(I)$; hence, using \cref{result:BV_W_s1_RL}, for all $s\in(0,1)$ we get 
\begin{equation} \label{eq:comm_D_s_I_Sobolev}
		D^{s}_{a+}[u](x)=\frac{u(a)}{\Gamma(1-s)}\frac{1}{(x-a)^s}+I^{1-s}_{a+}[u'](x),
\end{equation}
		where $u'$ denotes the weak derivative of $u$.
		Therefore, again by Proposition \ref{prop:cont_BLNT} and Remark \ref{rem:constant_Riesz_Thorin} we get
\begin{equation*}
	 \|D^s_{a+}[u]\|_{L^p(I)}\leq \frac{|u(a)|}{\Gamma(1 - s)} \left (\int_a^b\frac{dx}{(x-a)^{sp}}\right)^{\frac{1}{p}}+\frac{(b-a)^{1-s}}{\Gamma(2-s)}\left\|u'\right\|_{L^p(I)}.
		\end{equation*}
		Now, since $sp<1$, the first term in the right hand side is finite, and we obtain
\begin{equation}
		\label{eq:classici}
		\left\|D^s_{a+}[u]\right\|_{L^p(I)}\leq \frac{(b - a)^{\frac{1}{p} - s}}{(1 - sp)^{\frac{1}{p}} \Gamma(1 - s)}|u(a)|+\frac{(b-a)^{1-s}}{\Gamma(2-s)}\left\|u'\right\|_{L^p(I)}.
		\end{equation}
Thus, summing up \eqref{eq:continLp} and \eqref{eq:classici}, and exploiting Remark \ref{rem:Sobolev_AC_emb}, we deduce \eqref{eq:continuity_I_a_s_Sobolev}. Finally, if $u \in W^{1, p}_{a}(I)$, we have $u(a) = 0$, so that \eqref{eq:comm_D_s_I_Sobolev} reduces to $D^{s}_{a+}[u] = I^{1 - s}_{a+}[u']$. Therefore, Proposition \ref{prop:cont_BLNT} and Remark \ref{rem:constant_Riesz_Thorin} imply
\begin{equation*}
\|I^{1 - s}_{a+}[u]\|_{W^{1, p}(I)} = \|I^{1 - s}_{a+}[u]\|_{L^{p}(I)} + \|I^{1 - s}_{a+}[u']\|_{L^{p}(I)} \le \frac{(b-a)^{1-s}}{\Gamma(2-s)} \left ( \|u\|_{L^{p}(I)} + \|u'\|_{L^{p}(I)}\right ).
\end{equation*}
This concludes the proof.
	\end{proof}

\begin{corollary}
	\label{cor:inclusionofRLspaces}
	Let $1\le p\le q<\infty$ and $r,s\in(0,1)$ such that $sp<1$ and $r>s+\frac{1}{p'}$, where $p'$ denotes the H\"older conjugate of $p$. Then we have 
	$$
	W^{r,q}_{RL,a+}(I) \emb W^{s,p}_{RL,a+}(I).
	$$
	\end{corollary}
	\begin{proof}
		Since $W^{r,q}_{RL,a+}(I) \emb W^{r,p}_{RL,a+}(I)$, we are left to prove that $I^{1-s}_{a+}[u]\in W^{1,p}(I)$ for $u\in W^{r,p}_{RL,a+}(I)$. We notice that, thanks to Lemma \ref{lem:semigrouplaw},
		$$
		I^{1-s}_{a+}[u](x)=I^{r-s}_{a+}[I^{1-r}_{a+}[u]](x)=I^{1-\gamma}_{a+}[v](x),
		$$
		where $v:=I^{1-r}_{a+}[u]$ and $\gamma:=1-r+s$.
		Thanks to Proposition \ref{prop:regsobolev}, since $v\in W^{1,p}(I)$ we have $I^{1-\gamma}_{a+}[v]\in W^{1,p}(I)$ provided $\gamma p<1$, and this condition holds since $r>s+\frac{1}{p'}$.
		\end{proof}

\begin{remark} \label{rem:casopinfinito}
We notice that Proposition \ref{prop:regsobolev} covers the case $p=\infty$ if and only if $u(a)=0$. 
On the other hand, if $u(a)\neq 0$ we have neither continuous embedding, nor inclusion. Indeed, let us consider $I:=(0,1)$ and $u(x):=\cos(x)\in W^{1,\infty}(I)$. By \cref{result:BV_W_s1_RL}, we see that
	$$
	D^s_{0+}[u](x)=\frac{1}{\Gamma(1-s)}\left(\frac{1}{x^s}-\int_0^x\frac{\sin(t)}{(x-t)^s} \, dt\right).
	$$
We notice that $x^{-s}$ is not bounded when $x$ is close to 0, while it is easy to see that
	\begin{equation*}
\left | \int_{0}^{x} \frac{\sin(t)}{(x - t)^{s}} \, dt \right | \le	\int_{0}^{x} \frac{|\sin(t)|}{(x - t)^{s}} \, dt \le  \int_{0}^{x} \frac{1}{(x - t)^{s}} \, dt = \frac{x^{1 - s}}{1 - s} \le \frac{1}{1 - s}.
	\end{equation*}
	Thus, we conclude that $D^s_{0+}[u]\notin L^\infty(I)$.
\end{remark}

\begin{remark}
	We notice that the continuous embedding given by Corollary \ref{cor:inclusionofRLspaces} can be obtained as a byproduct of \cite[Theorem 31]{MR3144452}, which shows that the embedding is compact.
\end{remark}

Now, we show that the fractional integral actually improves the (weak) differentiability of a Sobolev function. To this purpose, we start with a simple remark.

\begin{remark}
	\label{rem:depending}
	Depending on the summability of a Sobolev function $u$ we notice that $I^{1-s}_{a+}[u]$ enjoys different improvements in regularity. In particular, we distinguish the case $p=1$ and the case $p>1$.
	\begin{enumerate}
		\item{Case $p=1$}
		
		If $u\in W^{1,1}(I)$, using Sobolev Embedding Theorem $u\in L^q(I)$ for any $1\le q\le\infty$, and so, thanks to Proposition \ref{prop:cont_BLNT}, $\displaystyle I^{1-s}_{a+}[u]\in\bigcap_{q>1/s}C^{0,s-\frac{1}{q}}(\overline{I})$.
		\item{Case $p>1$}
		
		If $u\in W^{1,p}(I)$, again by Sobolev Embedding Theorem, we have $u\in C^{0,1-\frac{1}{p}}(\overline{I})$.
		
		Using Proposition \ref{prop:holcontinuity}, for any $u\in C^{0,1-\frac{1}{p}}_{a}(\overline{I})$, we have 
		\begin{itemize}
			\item $\displaystyle I^{1-s}_{a+}[u]\in C^{0,2-s-\frac{1}{p}}_{a}(\overline{I})$ if $s+\frac{1}{p}>1$,
			\item $\displaystyle I^{1-s}_{a+}[u]\in H^{1,1}_{a}(\overline{I})$ if $s+\frac{1}{p}=1$,
			\item $\displaystyle I^{1-s}_{a+}[u]\in C^{1,1-s-\frac{1}{p}}_{a}(\overline{I})$ if $s+\frac{1}{p}<1$.
		\end{itemize}
		In the third case, it follows that $D^{s}_{a+}[u] \in C^{0,1-s-\frac{1}{p}}(\overline{I})$.
	\end{enumerate}
\end{remark}

Now, we are able to prove that when we apply the fractional integral $I^{1-s}_{a+}$ to a function in $W^{1,p}(I)$ for some $p>1$, we gain more differentiability. This means that the function $I^{1-s}_{a+}[u]$ belongs to a higher order fractional Sobolev space.

\begin{proposition}
	\label{prop:impregfrac}
	Let $p>1$ and $s\in\left(0,\min\left \{\frac{1}{p},\frac{p-1}{2p}\right \}\right)$. For all $u\in W^{1,p}_{a}(I)$, we have $I^{1-s}_{a+}[u]\in W^{s+1,p}(I)$.
	\begin{proof}
		We notice that the conditions $sp<1$ and $s+\frac{1}{p}<1$ are satisfied, and so, by Proposition \ref{prop:regsobolev} and Remark \ref{rem:depending}, we get $I^{1-s}_{a+}[u] \in W^{1,p}(I)\cap C^{1,1-s-\frac{1}{p}}_{a}(\overline{I})$.
		
		Now, we prove that $D^s_{a+}[u] \in W^{s,p}(I)$. Namely, we have to prove that
		$$
		\int_a^b\int_a^b\frac{|D^s_{a+}[u](x)-D^s_{a+}[u](y)|^p}{|x-y|^{sp+1}}dxdy<\infty.
		$$
		Now, we use the H$\"o$lder continuity of $D^s_{a+}[u]$ to say that
		$$
		|D^s_{a+}[u](x)-D^s_{a+}[u](y)|^p\leq C|x-y|^{p-sp-1},
		$$
		for some $C>0$ and for any $x,y\in I$.
		
		Therefore, we have
		$$
		\int_a^b\int_a^b\frac{|D^s_{a+}[u](x)-D^s_{a+}[u](y)|^p}{|x-y|^{sp+1}}dxdy\leq C\int_a^b\int_a^b\frac{1}{|x-y|^{2sp-p+2}}dxdy,
		$$
		where the integral on the right hand side converges since $s<\frac{p-1}{2p}$.
	\end{proof}
\end{proposition}

\begin{corollary}
	Let $p>1$, $s\in\left(0,\min\left \{\frac{1}{p},\frac{p-1}{2p}\right \}\right)$ and $u \in W^{1, p}_{a}(I)$. Then we have
	\begin{equation*}
	u\in W^{s,r}_{RL,a+}(I) \text{ for all } r\in\left[1,\frac{p}{1-sp}\right].
	\end{equation*}
	\begin{proof}
	By Proposition \ref{prop:impregfrac}, we have $I^{1-s}_{a+}[u]\in W^{s+1,p}(I)$, so that $D^{s}_{a+}[u] \in W^{s, p}(I)$. Hence, thanks to the fractional Sobolev Embedding, since $sp<1$, we have $D^s_{a+}[u]\in L^r(I)$ for all $r\in\left[1,\frac{p}{1-sp}\right]$. Therefore, $I^{1-s}_{a+}[u]$ belongs to $W^{1,r}(I)$ for all $r\in\left[1,\frac{p}{1-sp}\right]$, and this proves the claim. 
		\end{proof} 
\end{corollary}

\section{Sobolev-type embedding theorems for $W^{s,p}_{RL,a+}(I)$ and $BV^{s}_{RL,a+}(I)$} \label{sec:Sob_type_emb}

In this section we show a result analogous to the Sobolev embedding Theorem for Riemann-Liouville fractional Sobolev spaces. To the knowledge of the authors, this is an original result in this setting, which improves \cite[Proposition 7]{MR3311433}. We refer the reader {\it e.g.} \cite[Chapter 4]{MR2424078} for the classical Sobolev embedding Theorem for Sobolev spaces of integer order or \cite[Theorem 6.7]{MR2944369} for Sobolev spaces of fractional order.

\begin{theorem}[Riemann-Liouville fractional Sobolev embedding]
	\label{result:RLfracsobemb}
	Let $s\in(0,1)$, $1\le p \le\infty$ and $u\in W^{s,p}_{RL,a+}(I)$. We have the following cases:
	\begin{enumerate}
		\item if $p=1$, then $u\in L^{\frac{1}{1-s},\infty}(I)$, and in particular $u\in L^r(I)$ for any $r\in \left [1,\frac{1}{1-s} \right)$,
		\item if $p > 1$ and $u\in W^{s,p}_{RL,a+}(I) \cap I_{a+}^{s}(L^{1}(I))$, then 
		\begin{enumerate}
		\item if $1<p<\frac{1}{s}$, $u\in L^r(I)$ for any $r\in \left [1,\frac{p}{1-sp}\right ]$,
		\item if $sp=1$, $u\in L^r(I)$ for any $r\in[1,\infty)$,
		\item if $sp>1$, $u\in C^{0,\beta}(\overline{I})$ for any $\beta\in \left [0,s-\frac{1}{p} \right ]$.
		\end{enumerate}
	\end{enumerate}
	\end{theorem}
	\begin{proof}
If $u \in W^{s, 1}_{RL, a+}(I)$, by Lemma \ref{lem:invinteder} we have
\begin{equation*}
		u(x) = I^s_{a+}[D^s_{a+}[u]](x)+\frac{I^{1-s}_{a+}[u](a)}{\Gamma(s)}(x-a)^{s-1} \ \text{ for } \Leb{1}\text{-a.e. } x \in I.
\end{equation*}
It is clear that $(\cdot - a)^{s - 1} \in L^{\frac{1}{1 - s}, \infty}(I)$, while Lemma \ref{lem:weakpq} implies $I^s_{a+}[D^s_{a+}[u]] \in L^{\frac{1}{1 - s}, \infty}(I)$, since $D^s_{a+}[u] \in L^{1}(I)$. Then, it is enough to apply Lemma \ref{lem:embed_L_p_weak} to obtain point (1). Then, for $p > 1$ we use the additional assumption $u\in I^s_{a+}(L^1(I))$ together with Lemma \ref{lem:invinteder} to obtain the improved representation formula
\begin{equation*}
	u(x) = I^s_{a+}[D^s_{a+}[u]](x) \ \text{ for } \Leb{1}\text{-a.e. } x \in I.
\end{equation*}
In addition, $D^s_{a+}[u]\in L^p(I)$ since $u\in W^{s,p}_{RL,a+}(I)$. Therefore, the result follows directly from points 3, 4 and 5 of Proposition \ref{prop:cont_BLNT}.
	\end{proof}
	
\begin{remark}
	If we remove the $L^1$-representability assumption for the case $p > 1$ in Theorem \ref{result:RLfracsobemb}, we still obtain an improvement in summability if and only if $1 < p < \frac{1}{1 - s}$ and $s \in \left (\frac{1}{2}, 1 \right)$.
Indeed, if $u\in W^{s,p}_{RL,a+}(I) \setminus I_{a+}^{s}(L^{1}(I))$ for some $p > 1$, we clearly have $u \in W^{s,1}_{RL,a+}(I) \setminus I_{a+}^{s}(L^{1}(I))$, so that $u\in L^r(I)$ for any $r\in \left [1,\frac{1}{1-s} \right )$. On the other hand,  the term $(x - a)^{s - 1}$ in the representation formula \eqref{eq:invinteder_2} prevents us from obtaining a better result, since $(\cdot - a)^{s - 1} \in L^{r}(I)$ if and only if $1 \le r < \frac{1}{1 - s}$. Thus, we gain summability if and only if $1 < p < \frac{1}{1 - s}$. Let us now consider separately the cases $sp < 1$ and $sp \ge 1$. If $1 < p < \frac{1}{s}$, we obtain an improvement in summability if and only if $\left(p,\frac{1}{1-s}\right)\supseteq (p,\frac{1}{s})$. Then, it is enough to notice that
\begin{equation*}
\frac{1}{1 - s} > \frac{1}{s} \ \text{ if and only if } \ \frac{1}{2} < s < 1.
\end{equation*} 
On the other hand, in the case $sp \ge 1$, we have 
\begin{equation*}
\frac{1}{s} \le p < \frac{1}{1 - s},
\end{equation*}
so that it must be again $s \in \left (\frac{1}{2}, 1 \right)$.
\end{remark}

In addition, we can prove a similar results for functions with left Riemann-Liouville $s$-fractional bounded variation, which can be seen as the (one dimensional) ``Riemann-Liouville version'' of \cite[Theorem 3.8]{CS19}. 

\begin{theorem}\label{thm:Sob_emb_BV}
If $u \in BV^{s}_{RL, a+}(I)$, then $u\in L^{\frac{1}{1-s},\infty}(I)$, and in particular $u\in L^r(I)$ for all $r\in \left [1,\frac{1}{1-s} \right)$.
\end{theorem}
\begin{proof}
By Lemma \ref{lem:invinteder_BV}, any $u \in BV^{s}_{RL, a+}(I)$ satisfies
\begin{equation*}
		u(x) =  I^s_{a+}[\mathcal{D}^s_{a+}[u]](x)+\frac{I^{1-s}_{a+}[u](a+)}{\Gamma(s)}(x-a)^{s-1} \ \text{ for } \Leb{1}\text{-a.e. } x \in I.
\end{equation*}
It is immediate to check that $(\cdot - a)^{s - 1} \in L^{\frac{1}{1 - s}, \infty}(I)$, while Lemma \ref{lem:weak_meas} implies $$I^s_{a+}[\mathcal{D}^s_{a+}[u]] \in L^{\frac{1}{1 - s}, \infty}(I),$$ since $\mathcal{D}^s_{a+}[u] \in \mathcal{M}(I)$. Then, it is enough to apply Lemma \ref{lem:embed_L_p_weak} to conclude the proof. 
\end{proof}

We notice here that the embedding in Theorem \ref{result:RLfracsobemb} is sharp. The continuity of the fractional integral $I^s_{a+}$ from $L^p(I)$ into $L^r(I)$, with $1<p<\frac{1}{s}$ and $1\le r \le\frac{p}{1-sp}$ has been proved by Hardy and Littlewood in \cite[Theorem 4]{MR1544927}, but in the limiting cases $p=1$ and $p=\frac{1}{s}$ the continuity fails, as shown by the following examples.

\begin{example}
	Let $s\in(0,1)$, $1<\beta\le 2-s$, $I=(0,1)$ and
\begin{equation*}
	f(x)=\begin{cases}\displaystyle
	\frac{1}{x|\log(x)|^\beta}\ & \text{if}\quad 0<x\le \frac{1}{2} \\
	0 \ & \text{if}\quad \frac{1}{2} <x<1.
	\end{cases}
\end{equation*}
	Let $u:=I^s_{0+}[f]$. Clearly, $u\in I^s_{0+}(L^1(I))=I^s_{0+}(L^1(I))\cap W^{s,1}_{RL,0+}(I)$ since $f\in L^1(I)$. However, for all $x \in (0, 1/2)$, we have
	$$
	u(x)=\frac{1}{\Gamma(s)}\int_0^x\frac{dt}{t|\log(t)|^\beta(x-t)^{1-s}}>\frac{x^{s-1}}{\Gamma(s)}\int_0^x\frac{dt}{t|\log(t)|^\beta}=\frac{1}{\Gamma(s)(\beta-1)}x^{s-1}|\log(x)|^{1-\beta},
	$$
	so that
	$$
	\int_0^1 |u(x)|^{\frac{1}{1-s}}dx \ge\int_0^{\frac{1}{2}} |u(x)|^{\frac{1}{1-s}}dx> c(\beta, s) \int_0^{1/2}\frac{dx}{x|\log(x)|^{\frac{\beta-1}{1-s}}}=+\infty,
	$$
	since $\displaystyle\frac{\beta-1}{1-s}\le 1$. Thus, $u\notin L^{\frac{1}{1-s}}(I)$.
\end{example}

\begin{example}
	Let $s\in(0,1)$, $I:=(0,1)$ and
\begin{equation*}
	f(x)=\begin{cases}
	0 & \text{ if } 0<x< \frac{1}{2}, \\
	\displaystyle \frac{1}{(1-x)^s|\log(1-x)|} & \text{ if } \frac{1}{2} \le x<1.
	\end{cases}
\end{equation*}
	Now let $u:=I^s_{0+}[f]$; since $f\in L^{1/s}(I)$, by Lemma \ref{lem:representcriteri} we have $u\in I^s_{0+}(L^{1/s}(I)) \emb W^{s,1/s}_{RL,0+}(I)\cap I^s_{0+}(L^1(I))$. Now, we notice that
	$$
	\lim_{x\to 1^-} u(x)=\frac{1}{\Gamma(s)}\int_{1/2}^1\frac{dt}{(1-t)|\log(1-t)|}=+\infty,
	$$
	which implies that $u\notin L^{\infty}(I)$.
\end{example}

\section{Higher order fractional derivatives} \label{sec:higher}

In this  section, we point out that some of the results presented in the paper can be extended to higher order fractional derivatives.
\begin{definition}
	\label{def:rlecaphigherorder}
	Let $k\in\N$, $s\in(k-1,k)$ and $u$ be such that the fractional integrals $I^{k-s}_{a+}[u]$ and $I^{k-s}_{b-}[u]$ are sufficiently smooth. We define the Riemann-Liouville $s$-fractional derivatives of $u$ as
	$$
	D^s_{a+}[u](x):=\frac{d^k}{dx^k}I_{a+}^{k-s}[u](x).
	$$
	$$
	D^s_{b-}[u](x):=(-1)^k\frac{d^k}{dx^k}I_{b-}^{k-s}[u](x).
	$$
\end{definition}
We observe that $I^{k-s}_{a+}$ and $I^{k - s}_{b-}$ are the left and right fractional integral operators, respectively, as defined in \eqref{eq:frac_int_1} and \eqref{eq:frac_int_2}, since $k - s \in (0, 1)$ for all $k \in \N$ and $s \in (k - 1, k)$.
From this definition, for $u\in C^k(\overline{I})$, we immediately obtain a definition for higher order Caputo fractional derivatives:
$$
{}^CD^s_{a+}[u](x) := \frac{1}{\Gamma(k-s)}\int_a^x\frac{u^{(k)}(t)}{(x-t)^{s-k+1}}dt = D^s_{a+}[u](x)-\sum_{j=0}^{k-1}\frac{u^{(j)}(a)}{\Gamma(j-s+1)}(x-a)^{j-s},
$$
and
$$
{}^CD^s_{b-}[u](x) := \frac{(-1)^k}{\Gamma(k-s)}\int_x^b\frac{u^{(k)}(t)}{(t-x)^{s-k+1}}dt = D^s_{b-}[u](x)-\sum_{j=0}^{k-1}(-1)^j\frac{u^{(j)}(b)}{\Gamma(j-s+1)}(b-x)^{j-s}.
$$
These higher order fractional derivatives allow to define, for $p\ge 1$, $k\in\N$ and $s\in(k-1,k)$, {\em higher order left Riemann-Liouville fractional Sobolev spaces}
\begin{equation*}
W^{s,p}_{RL,a+}(I) := \left \{ u\in W^{k-1,p}(I) : I^{k-s}_{a+}[u]\in W^{k,p}(I) \right \}.
\end{equation*}

\begin{proposition}[Continuity of the fractional integral in higher order Sobolev spaces] \label{prop:continuity_higher_order}
If $k\ge 2$, $1\le p<\infty $ and $s\in \left (k-1,k-1+\frac{1}{p} \right )$, then the fractional integral $I^{k-s}_{a+}$ is a continuous operator from $W^{k,p}(I)\cap W_a^{k-1,p}(I)$ into $W^{k,p}(I)$. Moreover, for $k\ge 1$ and for all $s\in(k-1,k)$, $I^{k-s}_{a+}$ is a continuous operator
\begin{enumerate}
		\item from $W_a^{k,p}(I)$ into $W^{k,p}(I)$ for all $p \in [1, \infty]$,
		\item from $W_a^{k,1}(I)$ into $W^{k,\frac{1}{1-k+s},\infty}(I)$, and so into $W^{k,r}(I)$, for all $r\in \left [1,\frac{1}{1-k+s} \right )$,
		\item from $W_a^{k,p}(I)$ into $W^{k,r}(I)$ for all $p\in \left (1,\frac{1}{k-s} \right )$ and $r\in \left [1,\frac{p}{1-(k-s)p} \right ]$,
		\item from $W_a^{k,p}(I)$ into $C^{k,k-s-\frac{1}{p}}(\overline{I})$ for all $p \in \left (\frac{1}{k-s}, \infty \right )$,
		\item from $W_a^{k,\frac{1}{k-s}}(I)$ into $W^{k,r}(I)$ for all $r\in[1,\infty)$,
		\item from $W_a^{k,\infty}(I)$ into $C^{k,k-s}(\overline{I})$,
	\end{enumerate}
where $W^{k,\frac{1}{1-k+s},\infty}(I):=\left\{u\in W^{k, 1}(I) \cap W^{k-1,\infty}(I) : u^{(k)}\in L^{\frac{1}{1-k+s},\infty}(I)\right\}$.
	\end{proposition}
	\begin{proof}
We recall that $W^{k,p}(I)\emb AC^{k - 1}(\overline{I})$, for all $1\le p\le\infty$ and $k\in\N$, where $u \in AC^{k - 1}(\overline{I})$ if $u \in C^{k - 1}(\overline{I})$ and $u^{(k - 1)} \in AC(\overline{I})$. Hence, we see that the following representation formula obtained via iterated integrations by parts holds:
		\begin{equation*}
		I^{k-s}_{a+}[u](x)=\frac{1}{\Gamma(k-s)}\left(c_{s,k,k}\int_a^x u^{(k)}(t)(x-t)^{2k-s-1}dt+\sum_{i=0}^{k-1}c_{s,k,i}u^{(i)}(a)(x-a)^{k-s+i}\right),
		\end{equation*}
		where
\begin{equation*}
		c_{s,k,h}:=\displaystyle\begin{cases}
		1 & \text{ if }\quad h=0, \\
		\left(\prod_{l=0}^{h-1}(k-s+l)\right)^{-1} & \text{ if }\quad h\ge 1.
		\end{cases}
\end{equation*}
		Therefore, it is not difficult to check that $I^{k-s}_{a+}[u]$ admits weak derivatives in $L^p(I)$ up to order $k$ if $u$ vanishes in $a$ with all its derivatives up to order $k-2$. Indeed, if $u \in W^{k,p}(I)\cap W_a^{k-1,p}(I)$, for all $j \in \{0, \dots, k\}$ we have
	\begin{equation} \label{eq:iterintegr}
		(I^{k-s}_{a+}[u])^{(j)}(x)= \frac{d_{j, k, s}}{\Gamma(k-s)} \int_a^x u^{(k)}(t)(x-t)^{2k-s-1-j}dt + \frac{e_{j, k, s}}{\Gamma(k-s)} u^{(k - 1)}(a)(x-a)^{2k-s-1-j},
	\end{equation}	
	where $$d_{j, k, s} = c_{s,k,k} \frac{\Pi_{i = 0}^{j} (2 k - s - i)}{2k - s} \text{ and } e_{j, k, s} = c_{s,k,k - 1}  \frac{\Pi_{i = 0}^{j} (2 k - s - i)}{2k - s}.$$
	It is clear that the second term belongs to $L^{p}(I)$ for all $j \in \{0, \dots, k\}$ if and only if $s < k - 1 + \frac{1}{p}$. As for the first term, we notice that
	\begin{align*}
	 \frac{1}{\Gamma(k-s)} \left |\int_a^x u^{(k)}(t)(x-t)^{2k-s-1-j}dt\right| & \le \frac{1}{\Gamma(k-s)} (b -a)^{k - j} \int_a^x \frac{|u^{(k)}(t)|}{(x-t)^{1 - k + s}}dt \\
	 & = (b - a)^{k - j} I^{k - s}_{a+}[|u^{(k)}|](x),
	\end{align*}
	so that Proposition \ref{prop:cont_BLNT} and Remark \ref{rem:constant_Riesz_Thorin} imply
	\begin{align*}
 \frac{1}{\Gamma(k-s)} \left ( \int_{a}^{b} \left |\int_a^x u^{(k)}(t)(x-t)^{2k-s-1-j}dt\right|^{p} dx \right )^{\frac{1}{p}} & \le (b - a)^{k - j} \left \|I^{k - s}_{a+}[|u^{(k)}|]\right \|_{L^{p}(I)} \\
 & \le  (b - a)^{k - j} \frac{(b-a)^{k - s}}{\Gamma(k + 1 -s)} \|u^{(k)}\|_{L^{p}(I)}.
	\end{align*}
Therefore, $(I^{k - s}_{a+}[u])^{(j)} \in L^p(I)$ for all $j \in \{0, \dots, k\}$, $1\le p<\infty $ and $s\in \left (k-1,k-1+\frac{1}{p} \right )$. Thus, for all $u \in W^{k,p}(I)\cap W_a^{k-1,p}(I)$ we get
\begin{equation*}
\|I^{k-s}_{a+}[u]\|_{W^{k, p}(I)} \le C_{k, s, p, a,b} \left (  \|u^{(k)}\|_{L^{p}(I)} + |u^{(k-1)}(a)|\right ) \le \tilde{C} \|u\|_{W^{k,p}(I)},
\end{equation*}
thanks to Remark \ref{rem:Sobolev_AC_emb}, since $u^{(k-1)} \in W^{1, p}(I)$.
Furthermore, if $k \ge 1$ and $u\in W^{k,p}_a(I)$, we have $u^{(k-1)}(a)=0$, so that \eqref{eq:iterintegr} reduces to
\begin{equation}
\label{eq:reprformwkpa}
(I^{k-s}_{a+}[u])^{(j)}(x)= \frac{d_{j, k, s}}{\Gamma(k-s)} \int_a^x u^{(k)}(t)(x-t)^{2k-s-1-j}dt,
\end{equation}
for all $j \in \{0, \dots, k\}$.
In particular, if $j \in \{0, \dots, k-1\}$, then $(x-t)^{2k - s - 1 - j}$ is bounded for all $x, t \in (a, b)$ and $s \in (k-1,k)$; so that $I^{k-s}_{a+}[u] \in C^{k - 1}(\overline{I})$, with 
\begin{equation*}
\|I^{k-s}_{a+}[u]\|_{W^{k-1, \infty}(I)} \le C \|u^{(k)}\|_{L^{1}(I)} \le \tilde{C} \|u\|_{W^{k, p}(I)}
\end{equation*} 
for all $p \in [1, \infty]$ and $u \in W^{k,p}_a(I)$. 
On the other hand, if $j = k$, \eqref{eq:reprformwkpa} yields $$(I^{k-s}_{a+}[u])^{(k)} = d_{k,k,s} I^{k - s}_{a+}[u^{(k)}],$$ so that we can apply Proposition \ref{prop:cont_BLNT} to $u^{(k)}$ replacing $s$ with $k-s\in(0,1)$ in order to conclude the continuity of $I^{k-s}_{a+}$ from $W^{k,p}_{a}(I)$ into suitable Sobolev or H\"older spaces, depending on the values of $k, s, p$.
	\end{proof}

\begin{remark}
The first part of Proposition \ref{prop:continuity_higher_order} in the case $k=1$ is covered by Proposition \ref{prop:regsobolev}, where a homogeneous initial condition is not necessary to prove the continuity of the $(1-s)$-fractional integral from $W^{1,p}(I)$ to $W^{1,p}(I)$.
\end{remark}

\begin{remark}
Using the same counterexample of Remark \ref{rem:casopinfinito} in the case $k=1$, we see that, in the case $p=\infty$, homogeneous conditions in the initial point for all the derivatives up to order $k-1$ are necessary in order to show that $I^{k-s}_{a+}[u] \in W^{k,\infty}(I)$.
\end{remark}

The introduction of higher order Riemann-Liouville fractional Sobolev spaces allows us to prove the following proposition involving the space $$BH(I):=\left\{u\in W^{1,1}(I)\,|\,u'\in BV(I)\right\},$$ which is known in the literature as the space of {\em functions with bounded Hessian} in $I$. Originally introduced in \cite{MR746501}, $BH$ is the natural setting for second order variational problems with linear growth (see {\it e.g.} \cite{MR2118417} for applications in image analysis). For our purposes, we consider the subspace $BH_{a}(I)$; that is, the space of functions $u \in BH(I)$ such that $u(a) = 0$, which is well defined, since $u \in AC(\overline{I})$, being a Sobolev function.

\begin{proposition}
	Let $u\in BH_{a}(I)$, then $u\in W^{s,1}_{RL,a+}(I)$ for all $s\in (1,2)$.
	\end{proposition}
	\begin{proof}
		By definition, $u\in W^{1,1}(I)$ and $u'\in BV(I)$. Therefore, thanks to Theorem \ref{result:BV_W_s1_RL}, we have $u'\in W^{\sigma,1}_{RL,a+}(I)$ for all $\sigma\in(0,1)$, and so $I^{1-\sigma}_{a+}[u']={}^CD^{\sigma}_{a+}[u]\in W^{1,1}(I)$. Now, since $u(a)=0$, we have ${}^C D^{\sigma}_{a+}[u](x)=D^{\sigma}_{a+}[u](x)$ for all $x\in I$ by \eqref{eq:RLecap}, since $u \in W^{1, 1}(I)$ implies the existence of a representative of $u$ in $AC(\overline{I})$. This implies that $I^{1-\sigma}_{a+}[u]\in W^{2,1}(I)$ for all $\sigma\in (0,1)$. Now, if we set $\sigma:=s-1$ for $s\in(1,2)$, the claim plainly follows.
	\end{proof}

\section{Open Problems} \label{sec:open}

As noticed in Remark \ref{rem:mironsic}, we are not able to prove (or disprove) that for $s\in(0,1)$ and $p>1$ the inclusion
\begin{equation}
\label{eq:conjectureinclusion}
W^{s,p}(I)\emb W^{s,p}_{RL,a+}(I)
\end{equation}
holds. 

In addition, Proposition \ref{prop:newprop} does not cover the case $r\in \left [s,s+\frac{1}{p'} \right )$. In any case, we think that, in order to prove the inclusion in \eqref{eq:conjectureinclusion}, the condition $sp<1$ is essential.
	
	 Indeed, if $sp<1$, thanks to Remark \ref{rem:densitàdellecc1}, the set $C^1_c(I)$ is dense in $W^{s,p}(I)$. Therefore, firstly one should be able to prove an analogous of Lemma \ref{lem:density} for functions in $W^{s,p}(I)$, and once proved that
\begin{equation*}
D^s_{a+}[u](x)=\frac{1}{\Gamma(1-s)}\frac{u(x)}{(x-a)^s}+\frac{s}{\Gamma(1-s)}\int_a^x\frac{u(x)-u(t)}{(x-t)^{s+1}}dt \  \text{ for } \Leb{1}\text{-a.e. } x \in I,
\end{equation*}
one could estimate the $L^p$ norm of the first term in the right-hand side thanks to the fractional Hardy inequality (Lemma \ref{lem:hardy}). However, it is not yet clear to us how to handle the second term. Indeed, thanks to H\"older's inequality, a slightly rough estimate yields
\begin{equation*}
\int_a^b \left|\int_a^x\frac{u(x)-u(t)}{(x-t)^{s+1}} \, dt\right|^p dx \le (b - a)^{p - 1} \int_a^b \int_a^x\frac{|u(x)-u(t)|^p}{|x-t|^{sp+p}}\, dt \, dx = (b - a)^{p - 1} \left (S_1+S_2 \right ),
\end{equation*}
where
\begin{equation*}
S_1:=\int_a^b \int_{I_1(x)}\frac{|u(x)-u(t)|^p}{|x-t|^{sp+p}} \, dt \, dx, \ I_{1}(x) := \{ t \in (a, b) : |x - t| > 1 \},
\end{equation*}
and
\begin{equation*}
S_2:=\int_a^b \int_{I_2(x)}\frac{|u(x)-u(t)|^p}{|x-t|^{sp+p}} \, dt \, dx, \ I_{2}(x) := \{ t\in(a,b) : |x-t|\le 1\}.
\end{equation*}
As for $S_1$, we see that
\begin{equation*}
S_1\le \int_a^b \int_{I_{1}(x)}\frac{|u(x)-u(t)|^p}{|x-t|^{sp+1}} \, dt \, dx \le [u]^p_{W^{s,p}(I)}.
\end{equation*}
However, we are not able to prove (or disprove) the existence of a constant $C = C(s, p, I) > 0$ such that an estimate of the form
\begin{equation*}
S_2\leq C[u]^p_{W^{s,p}(I)},
\end{equation*}
or its weaker formulation
\begin{equation*}
S_2\leq C\left\|u\right\|^p_{W^{s,p}(I)},
\end{equation*}
holds true.


\end{document}